\documentclass[reqno]{article}
\usepackage{amsmath,amsfonts,amssymb,amsthm,cite}
\newtheorem{theorem}{Theorem}
\newtheorem{lemma}[theorem]{Lemma}

\newtheorem{claim}[theorem]{Claim}
\theoremstyle{definition}
\newtheorem{remark}{Remark}
\newtheorem{definition}{Definition}

\newcommand\eps{\varepsilon}
\renewcommand\le{\leqslant}
\renewcommand\ge{\geqslant}
\newcommand\cc{{\mathrm{c}}}
\newcommand\E{\operatorname{\mathbb E{}}}
\renewcommand\Pr{\operatorname{\mathbb P{}}}
\newcommand\pto{\overset{\mathrm{p}}{\to}}

\newcommand\dd{{\mathrm d}}
\newcommand\dx{{\dd x}}
\newcommand\bb[1]{\bigl(#1\bigr)}

\newcommand\floor[1]{\lfloor #1 \rfloor}

\newcommand\symdiff{\bigtriangleup}
\newcommand\cF{\mathcal{F}}

\newcommand\cS{{\mathcal S}}
\newcommand\cT{{\mathcal T}}
\newcommand\cC{{\mathcal C}}

\newcommand\tA{\widetilde{A}}
\newcommand\tE{\widetilde{E}}
\newcommand\tT{\widetilde{T}}
\newcommand\tN{\widetilde{N}}
\newcommand\tC{\widetilde{C}}
\newcommand\tS{\widetilde{S}}
\newcommand\tX{\widetilde{X}}
\newcommand\tY{\widetilde{Y}}
\newcommand\hI{\hat{I}}
\newcommand\tmu{\widetilde{\mu}}
\newcommand\bm{\mathbf{m}}
\newcommand\tc{t_{\mathrm{c}}}
\newcommand\good{W}
\newcommand\hg{\hat{g}}
\newcommand\hF{\hat{F}}
\newcommand\dt{{\mathrm{d}t}}
\newcommand\dnu{{\mathrm{d}\nu}}
\newcommand\kmax{k_{\mathrm{max}}}
\newcommand\op{o_{\mathrm{p}}}
\newcommand\Op{O_{\mathrm{p}}}
\newcommand\Var{\mathrm{Var}}%

\begin{document}
\title{Counting subgraphs in bounded-size Achlioptas processes}%
\author{Mihyun Kang\thanks{Institute of Discrete Mathematics, Graz University of Technology, Steyrergasse 30, 8010 Graz, Austria.
E-mail: {\tt kang@math.tugraz.at}. Supported in part by the Austrian Science Fund (FWF) [10.55776/F1002].}\ \ and Oliver Riordan%
\thanks{Mathematical Institute, University of Oxford, Andrew Wiles Building, Woodstock Road, Oxford OX2 3GG, UK. E-mail: {\tt oliver.riordan@maths.ox.ac.uk}.}%
\ \thanks{For the purpose of open access, the authors have applied a CC BY public copyright licence to any author accepted manuscript arising from this submission.}}%
\date{May 9, 2026}%

\maketitle

\begin{abstract}
Achlioptas processes such as the Bohman--Frieze process are much harder to analyse 
than the classical Erd\H os--R\'enyi process, due to the dependence between
edges added at different stages. This dependence means that most analysis
so far is dynamic, often based on the differential equation method. In the Erd\H os--R\'enyi
case there is an alternative static approach,
pioneered by Erd{\H o}s, R\'enyi and Bollob\'as, 
based on evaluating the expectation (and higher moments)
of various subgraph counts, and using this to study the component structure.
Here we show that this latter approach can be applied (with some complications)
to the Bohman--Frieze process. For example, we are able to show that the expected number $\mu_{k,t,n}$ of $k$-vertex tree components after $tn$ steps satisfies (essentially)
$\mu_{k,t,n}=c_{k,t}n(1+O(k/\sqrt{n}))$.
Our method
gives a very complicated formula for $c_{k,t}$, which seems to be unusable. However,
since $c_{k,t}$ does not depend on $n$, we may use recent results obtained
by the differential equation method and branching process analysis to find the asymptotics
of $c_{k,t}$ as $k\to\infty$. The latter results also give a formula
for $\mu_{k,t,n}$ of the form $c_{k,t}n$
plus an error term,
with a much more usable description of $c_{k,t}$ but a much worse error term. We
combine the best of both worlds to prove a number of new results about the process near
criticality. In particular, we obtain extremely sharp bounds on the size of the
largest non-giant component near criticality, including the limiting distribution
of its fluctuations.
\end{abstract}%

\section{Introduction}
We study the evolution of random graph processes which are based on the paradigm
of the power of two choices. The processes we consider begin with an empty graph
on $n$ vertices. In each subsequent step two potential edges are chosen randomly,
and we select one edge to be included in the evolving graph. Such processes were suggested by Dimitris Achlioptas around 2000, and are now called
\emph{Achlioptas processes}. An important special case is the class of
\emph{bounded-size rules}, where the decision which of the offered edges to select depends
only on the sizes of the (up to four) components involved, with all sizes
above some constant being treated the same.
The phase transition in Achlioptas processes,
in particular for bounded-size rules, has received considerable
attention~\cite{BBWcrt,BBWsub,BBWagg,BF01,BKgiant,DKP,ELMPSconn,JS12,KPconn,KPS,KPSconjecture,KLSsub,KLSHam,RWsci,RW12,RWself,RW13,RWsub,RW17,JRW18,SS,SW07}.
In this paper we consider the following (minor variant of the) Bohman--Frieze process,
a simple example of a bounded-size rule: start with
an empty graph $G_0$ on $n$ vertices, say with vertex set $[n]=\{1,2,\ldots,n\}$.
At step $m$, i.e., given $G_{m-1}$,
pick two `potential edges' $e_{m,1}$ and $e_{m,2}$ independently and uniformly from
all $\binom{n}{2}$ possible edges. If $e_{m,1}$ joins two vertices that are isolated in $G_{m-1}$,
then set $e_m=e_{m,1}$; otherwise set $e_m=e_{m,2}$. In either case, set $G_m=G_{m-1}\cup \{e_m\}$.
Note that multiple edges are possible but, since we will always consider $m=O(n)$,
there will be rather few of them, and they turn out not to matter.
One could consider variants where $e_{m,1}$ and $e_{m,2}$ are chosen from edges
not present in $G_{m-1}$, and perhaps with $e_{m,2}\ne e_{m,1}$.
Our methods apply to these variants with appropriate minor modifications.
Note that the sequence $(G_m)$ of course depends on $n$; when necessary we
write $G_m^{(n)}$ to make this clear.%

\begin{definition}[Asymptotic notation]
Throughout the paper we use the following standard asymptotic notation: $\pto$ denotes convergence
in probability. For a sequence of random variables $X_n$ and a deterministic function $f(n)$,
we write $X_n=\Op(f(n))$ if $X_n/f(n)$ is bounded in probability, and $X_n=\op(f(n))$ if $X_n/f(n)\pto 0$.
If $E_n$ is a sequence of events (often written without the explicit dependence on $n$), we say
that $E_n$ holds \emph{with high probability} or \emph{whp} if $\Pr(E_n)\to 1$ as $n\to\infty$. 
All asymptotics in this paper are as  $n\to\infty$.
\end{definition}

Spencer and Wormald~\cite{SW07} showed
using the differential equation method that, for any bounded-size rule,
there are functions $\rho_k(t)$ such that,
for any  $k\ge 1$ and $t\in [0,\infty)$,
\begin{equation}\label{conv}
 N_k(G_{\floor{tn}})/n \pto \rho_k(t)
\end{equation}
as $n\to\infty$, where $N_k(G)$ denotes the number of vertices of a graph $G$
in $k$-vertex components.
From this and a result of Warnke and the second author~\cite[Theorem 3]{RW12}
it follows that there exists a continuous, increasing function $\rho: [0,\infty)\to [0, 1]$ such that for any $t\in [0,\infty)$ we have
\begin{equation}\label{scalinglimit}
L_1(G_{{\floor{tn}}})/n \pto \rho(t),
\end{equation}
where $L_j(G)$ denotes the number of vertices in the $j$th largest component of a graph $G$ for $j\in \mathbb N$.
The function $\rho(t)$ is zero up to a certain `critical time' $\tc$, while there is a positive constant $\xi$ such
that
\begin{equation}\label{rhosim1}
 \rho(\tc+\eps) \sim \xi\eps
\end{equation}%
as $\eps\to 0$ from above; see~\cite{JS12,DKP} or~\cite[Theorem 2.5]{RW17}, where more precise results are proved.
The constant $\xi$ is twice the constant $\gamma$ appearing in Theorem 3.5 of Janson and Spencer~\cite{JS12},
the factor of 2 coming from the different normalization of the number of steps.%

Building on results of theirs with Janson~\cite{JRW18}, Riordan and Warnke~\cite{RW17} established, among many other results, the following asymptotics for any bounded-size $\ell$-vertex rule, which includes the Bohman--Frieze process. %
\begin{theorem}[Theorem 2.9,~\cite{RW17}]\label{RW}\footnote{The statement is somewhat simplified from that in~\cite{RW17}, since here we only consider the Bohman--Frieze process. Hence the set ${\cal S}_{\cal R}$ of `reachable' component
sizes consists of all positive integers.}
There exist a constant $\eps_0>0$ and non-negative analytic functions  $\gamma(t)$ and $\delta(t)$ on
\begin{equation}\label{I0def}
 I_0=[\tc-\eps_0,\tc+\eps_0]
\end{equation}
such that
\[
 \rho_k(t) =(1+O(1/k))  \gamma(t) k^{-3/2} e^{-\delta(t) k},
\]
uniformly in $k\ge 1$ and $t\in I_0$, with  $\gamma(\tc), \delta''(\tc)>0$ and $\delta(\tc)=\delta'(\tc)=0$. 
\end{theorem}%

Our main result is the following; here and in the rest of the paper we consider only the Bohman--Frieze
process, although almost all our arguments (and we expect all of our results) extend to general
bounded-size rules.%
\begin{theorem}\label{th1}
Let $k=k(n)=o(\sqrt{n})$ and $m=m(n)\in nI_0$ be integers, where $I_0$ is as in \eqref{I0def}.
There is a `good' event $\good=\good_n$
and a random variable $\tN_k=\tN_{k,n,m}$ such that $\Pr(\good)\to 1$ as $n\to\infty$,
and when $\good$ holds, then $N_k(G_m^{(n)})=\tN_k$. Moreover,
\[
 \E \tN_k \sim n \rho_k(m/n),
\]
and if $k_1(n)\le k_2(n)=o(\sqrt{n})$ are chosen so that $\lambda(n)=\sum_{k_1\le k<k_2} n\rho_k(m/n)/k\to\infty$,
then $X=\sum_{k_1\le k<k_2} N_k(G_m^{(n)})/k$ satisfies $X/\lambda(n)\pto 1$.
\end{theorem}%
Roughly speaking, this result says that,
off a `global' bad event of probability $o(1)$, the number of $k$-vertex 
components is as we would expect from the differential equation asymptotics.
In fact, we shall prove a slightly more detailed version of Theorem~\ref{th1};
see Lemmas~\ref{lsofar} and~\ref{lmain}.
These more detailed results say
that for components of size $o(\sqrt{n})$, off our bad event,
the behaviour is essentially as if the number of tree
components of each size $k$ is Poisson with mean around $n\rho_k(m/n)/k$, with these numbers
independent, and with many fewer non-tree components.
Using this lemma and the component-size-gap idea of Bollob\'as~\cite{BB84}, we obtain
the following very precise bounds on the size of the largest component in much of the
subcritical regime, giving not only the asymptotics (also established in~\cite{RW17}), but 
also the limiting distribution of the fluctuations.
%
\begin{theorem}[Subcritical phase]\label{th_subcrit}
Let $m=m(n) \in \mathbb N$. Set $t=t(n)=m/n$ and $\eps=\eps(n)=t-\tc$, and suppose that $-\eps_0\le \eps<0$, 
where $\eps_0$ is as in Theorem~\ref{RW}, and
\begin{equation}\label{epscond}
 |\eps| n^{1/4} (\log n)^{-1/2} \to\infty.
\end{equation}
For any constant $x$ we have
\begin{equation}\label{L1bd}
 \Pr\left( L_1(G_m^{(n)}) \le
 \delta(t)^{-1}\left(\log(|\eps|^3n)-\frac{5}{2}\log\log(|\eps|^3n)+c(t)+x\right)  \right) \to e^{-e^{-x}}
\end{equation}
as $n\to\infty$, where $c(t)$ is a bounded function of $t$ defined in \eqref{ctdef}, and $\delta(t)$ is as
in Theorem~\ref{RW}. In particular,%
\[
 L_1(G_m^{(n)})\frac{\delta(t)}{\log(|\eps|^3n)} \pto 1.
\]
\end{theorem}%
In the supercritical case we have a corresponding result for the second largest component,
with a slightly stronger condition on $\eps$. Moreover, we also
obtain good bounds on the size of the giant component when $\eps^6n\to\infty$.   
\begin{theorem}[Supercritical phase]\label{th_supcrit}
Let $m=m(n) \in \mathbb N$. Set $t=t(n)=m/n$ and $\eps=\eps(n)=t-\tc$, and suppose that $0<\eps\le\eps_0$,
where $\eps_0$ is as in Theorem~\ref{RW}, and $\eps^6 n\to\infty$.
For any constant $x$ we have
\begin{equation}\label{L2bd}
 \Pr\left( L_2(G_m^{(n)}) \le
 \delta(t)^{-1}\left(\log(\eps^3n)-\frac{5}{2}\log\log(\eps^3n)+c(t)+x\right)  \right) \to e^{-e^{-x}}
\end{equation}
as $n\to\infty$, where $c(t)$ is a bounded function of $t$ defined in \eqref{ctdef}, and $\delta(t)$ is as in Theorem~\ref{RW}.
Furthermore,
\begin{equation}\label{newL1bd}
 L_1(G_m^{(n)}) = \rho(t)n+\Op(\eps^{-1/2}n^{3/4}) = (1+\op(1))\rho(t)n.
\end{equation}
\end{theorem}%
In interpreting the bounds in \eqref{newL1bd}, it may help to note that $\rho(\tc+\eps)=\Theta(\eps)$ for $\eps>0$;
see \eqref{rhosim1}.
Since we assume $\eps^6n\to\infty$,
the first expression for $L_1$ in \eqref{newL1bd} thus implies the second. The second, weaker bound was
proved by Warnke and the second author~\cite[Theorem 2.8]{RW17}, under the weaker assumption $\eps^3n\to\infty$.

\subsection{Relationship to earlier results}

There has been considerable interest in studying the near-critical behaviour of random graph processes
other than the classical Erd\H os--R\'enyi (ER) process, and in particular that of the Bohman--Frieze process,
as one of the simplest examples that does not have the independence properties of the ER process.
The behaviour inside the `scaling window' of the phase transition was established by Bhamidi, Budhiraja and Wang~\cite{BBWagg,BBWcrt},
but that outside, where $t=m/n=\tc+\eps$ for $\eps\to 0$ with $\eps^3n\to\infty$, is only partially understood.

As an example, Kang, Perkins and Spencer~\cite{KPS} conjectured (their Conjecture 1) 
that for the Bohman--Frieze process  $(G_m^{(n)})$ 
there exists a constant $K$ such that for any fixed $\eps>0$, we have $\Pr\bb{L_1(G_{(\tc-\eps)n}^{(n)})\le K \eps^{-2}\log n}\to 1$ as $n\to \infty$.
Bhamidi, Budhiraja and Wang~\cite{BBWsub} made some progress towards this, showing that for any bounded-size rule
and any $\beta\in (0,1/4)$, there exists $K\in (0,\infty)$ such that $\Pr\bb{L_1(G_{tn}^{(n)})\le K(\tc-t)^{-2}\log^4 n,\ \forall t\le\tc-n^{-\beta}}\to 1$. This bound was in fact a key ingredient in their analysis of
the behaviour inside the window in~\cite{BBWcrt}. Sen~\cite{SS} proved Conjecture 1 of~\cite{KPS},
and Warnke and the second author~\cite{RW17} proved a much tighter form
of the conjecture, obtaining a bound of the form
\begin{equation*} 
 L_1(G_m^{(n)}) =  \delta(\tc-\eps)^{-1}\left(\log(\eps^3n)-\frac{5}{2}\log\log(\eps^3n)+\Op(1)\right)
\end{equation*}
whenever $\eps=\tc-m/n$ satisfies $\eps \in (0,\eps_0)$ and $\eps^3 n \to \infty$ as $n \to \infty$.
Specifically for the Bohman--Frieze rule, and subject to an additional restriction on $\eps(n)$,
Theorem~\ref{th_subcrit} sharpens this result yet further, finding the limiting distribution of the $\Op(1)$ term.
This is especially significant since it is only here that, via the function $c(t)$, a (small!) difference
from the Erd\H os--R\'enyi behaviour is visible.
It might in principle be possible to push the first and second moment arguments of~\cite{RW17} to higher moments
and so obtain a result like Theorem~\ref{th_subcrit} by the methods of that paper, but the arguments for two moments there are already
extremely complicated, so this would likely be difficult.

Considering the size $L_2$ of the second largest component in the supercritical case,
the previously known bounds are \emph{much} weaker: we are not aware of anything at all comparable in accuracy
to \eqref{L2bd}. For example, in~\cite[Theorem 2.8]{RW17} only a (not very strong) upper bound is given.
A key advantage of our tree-counting methods is that they work equally well just
above and just below the critical point, so we obtain an extremely sharp bound on $L_2$ in the supercritical case.
As far as we are aware, no similarly sharp bounds have been shown for any type of random graph process
with dependence;
the problem is that any `duality principle' is much harder to understand precisely than in the Erd\H os--R\'enyi case.

Except for the restrictions \eqref{epscond} and $\eps^6n\to\infty$, Theorem~\ref{th_subcrit} and the bound \eqref{L2bd} in Theorem~\ref{th_supcrit} are direct analogues of the refinements
given by {\L}uczak~\cite{Luczak90}
of the 1984 results of Bollob\'as~\cite{BB84} for the Erd\H os--R\'enyi model. ({\L}uczak has $2\eps^{-2}$
in place of $\delta(t)^{-1}$ but, as noted in~\cite{rg_bp}, this is incorrect. For
the Erd\H os--R\'enyi model $\delta(t)= \eps^2/2 -\Theta(\eps^3)$ for $t=1+\eps$; see~\cite{BB84}.)
We need to impose condition \eqref{epscond} because, in their present form,
our tree-counting arguments only apply to trees of size up to $o(\sqrt{n})$.
To replace this condition with `what it should be' (namely, $\eps^3n\to\infty$) would require handling
all trees of size $o(n^{2/3})$. This may be possible in principle, but would certainly introduce many
additional complications; see Remark~\ref{rem_precise}.
It may perhaps be possible to avoid some of these complications with
branching process arguments, along the lines given by Bollob\'as and the second author in~\cite{rg_bp} or perhaps~\cite{BRsimple}. 

\begin{remark}
In this paper we concentrate on the Bohman--Frieze process. As noted in Section~\ref{ss_refined} (in a footnote),
our arguments extend with only notational complications to a class of bounded-size Achlioptas processes
satisfying a certain monotonicity property. Furthermore, that monotonicity is only used in one place,
to deal with a certain estimate, and is presumably not needed for our results. We have not investigated
this further for fear of burying the key ideas -- in particular that we can take the evolution
of the number of isolated vertices as a `background' and effectively recover enough
independence to apply tree-counting methods -- in superficial complications.
\end{remark}

\begin{remark}
The history of this paper and its relationship to the paper~\cite{RW17} by Warnke and the second author is rather unusual.
The ideas here are very different from those in~\cite{RW17}, and in fact the first draft of this
paper was written earlier, after Kang, Perkins and Spencer~\cite{KPS} appeared,
using~\cite[Theorem~3]{KPS} rather than
\cite[Theorem 2.9]{RW17} as an `input' (here Theorem~\ref{RW}). Unfortunately, as noted in~\cite{KPSconjecture},
the proof of Theorem~3 in~\cite{KPS} is seriously flawed, and does not appear to be fixable. Still,
it may well be that Theorem~\ref{RW}, which is purely a statement about the solution to certain differential
equations (DEs), has a pure DE proof. One might hope for such a proof that is considerably simpler
than the argument in~\cite{RW17}, which does use DE methods, but also uses comparisons to branching processes
and results from Janson, Riordan and Warnke~\cite{JRW18}, and is overall rather involved. If such a proof
is found, then it could be dropped in place of Theorem~\ref{RW} here, giving a simpler overall route to some
(but by no means all) of the results of~\cite{RW17}. In any case, although here we use results
from~\cite{RW17} in various places, these are mostly for convenience: the only essential dependence
on~\cite{RW17} is via Theorem~\ref{RW}, which is of course critical for our results.
\end{remark}

\subsection{The basic method}
A key property of the process $(G_m)$ is that the random edges $e_m$ are not independent,
but, given $G_m$, the probability that $e_{m+1}$ is some specific edge $e$ depends on $G_m$
in a rather simple way. To be specific, let $I_m$ denote the (random, of course) number of isolated
vertices in $G_m$. Then whenever both ends of $e$ are isolated in $G_m$ we have
\begin{eqnarray}
 \Pr(e_{m+1}=e \mid G_m) &=& \binom{n}{2}^{-1} + \left(1-\binom{I_m}{2}\binom{n}{2}^{-1}\right)\binom{n}{2}^{-1} \nonumber \\
 &=& 2n^{-2}\bb{2-(I_m/n)^2} +O(n^{-3}), \label{pisol}
\end{eqnarray}
since $e_{m+1}=e$ if and only if $e_{m+1,1}=e$ or $e_{m+1,1}$ does not join two isolated vertices and $e_{m+1,2}=e$.
Otherwise, we have
\begin{eqnarray}
 \Pr(e_{m+1}=e \mid G_m) &=& \left(1-\binom{I_m}{2}\binom{n}{2}^{-1}\right)\binom{n}{2}^{-1} \nonumber\\
 &=& 2n^{-2}\bb{1-(I_m/n)^2} +O(n^{-3}), \label{pnisol}
\end{eqnarray}
since for this to happen $e_{m+1,1}$ must not join two isolated vertices and we must have $e_{m+1,2}=e$.%

Note that the asymptotic expressions \eqref{pisol} and \eqref{pnisol} remain valid in the variants
of the method described at the start of the introduction,
as long as $m=O(n)$ and we consider $e$ not already present in $G_m$.%

The key observation is a very simple one: it is known (from the differential equation method; see~\cite{SW07}) that $I_m/n$ is concentrated around a deterministic trajectory. Hence, we can consider the conditional probability that a certain edge $e$ is added at step $m$ as essentially one of two deterministic
functions of $m$; the first function applies if both ends of $e$ are isolated, and the second otherwise.  
We can then use this to evaluate the probability that a certain small connected set of edges forms a component of $G_m$. This allows us to adapt tree-counting first- and second-moment arguments used by Erd\H os and R\'enyi~\cite{ER60}
and Bollob\'as~\cite{BB84} to study the Erd\H{o}s--R\'enyi model, although there are of course complications.
To implement this strategy we shall need a strong form of the concentration result of Spencer and Wormald~\cite{SW07} mentioned above. Although it seems likely that this (or something very similar) has already been proved by others, we have not found it in the literature, and include a proof for completeness. Here and for the rest of the paper we fix a function $\omega(n)$ with $\omega(n)\to\infty$ (slowly). When we come to consider components up to some size $k(n)$ that is $o(\sqrt{n})$ we may of course assume that $\omega(n)=o(\sqrt{n}/k(n))$, so
\begin{equation}\label{kosn}
 k=o(\sqrt{n}/\omega).
\end{equation}
\begin{lemma}\label{iconc}
Let $\rho_1(t)$ denote the solution to the differential equation
\begin{equation}\label{rho1}
 \frac{\dd \rho_1}{\dd t} = -2\rho_1^2 - 2(1-\rho_1^2)\rho_1,
\end{equation}
with initial condition $\rho_1(0)=1$. Let $C$ be a constant. Then with probability $1-o(1)$ we have
\begin{equation}\label{gooddef}
 \forall m\le Cn: |I_m/n-\rho_1(m/n)|\le \omega n^{-1/2}.
\end{equation}
\end{lemma}
\begin{proof}
We shall apply the Azuma--Hoeffding inequality, adapting the ideas of the proof of Theorem 5.1 of~\cite{W99}. The crucial idea is to
estimate $I_m/n-\rho_1(m/n)$ within an error $e(m/n)\omega n^{-1/2}$, by
choosing the `error' function $e(m/n)$ so that (at least until the point
that some unlikely `bad' event holds),
\begin{align}
I_m^+&= I_m-\rho_1(m/n)n + e(m/n)\omega n^{1/2} \rlap{\hbox{\quad and}} \label{Impdef}\\ 
I_m^-&= I_m-\rho_1(m/n)n - e(m/n)\omega n^{1/2} \nonumber
\end{align}
are a submartingale and a supermartingale, respectively, with respect to the natural filtration
$\mathcal F_0\subseteq\mathcal F_1\subseteq \cdots$ associated to the process $(G_m)$. To be totally
explicit, we set
\begin{equation}\label{edef}
 e(t)=e^{14t}.
\end{equation}
We first study the expected change in the number of isolated vertices
in step $m+1$, i.e., $\E[I_{m+1}-I_m\mid \cF_m]$. Of course, the number of isolated vertices can only decrease. If $e_{m+1,1}$ joins two isolated vertices,
an event of probability $\binom{I_m}{2}\binom{n}{2}^{-1}$, then $I_{m+1}-I_m=-2$.
Otherwise, $I_{m+1}-I_m$ is $-2$, $-1$ or $0$ according to whether $e_{m+1,2}$ joins two isolated vertices,
one isolated vertex and one other vertex, or two non-isolated vertices. Thus
\begin{align}
\E[I_{m+1}-I_m\mid\cF_m]
&= -2 \binom{I_m}{2}\binom{n}{2}^{-1}
-2 \left(1-\binom{I_m}{2}\binom{n}{2}^{-1}\right)\binom{I_m}{2}\binom{n}{2}^{-1}\nonumber\\
&\hskip 1.5cm -\left(1-\binom{I_m}{2}\binom{n}{2}^{-1}\right)I_m(n-I_m)\binom{n}{2}^{-1}\nonumber\\
&= -2 \left(\frac{I_m}{n}\right)^2 -2\left(1-\left(\frac{I_m}{n}\right)^2\right)\frac{I_m}{n}+O(n^{-1}).\label{trendhyp1}
\end{align}
(Of course, this formula can also be derived from \eqref{pisol} and \eqref{pnisol}.)
From \eqref{rho1} and \eqref{edef}, $\rho_1(t)$ and $e(t)$ are infinitely differentiable,
so $\rho_1''$ and $e''$ are continuous and hence bounded on $[0,C]$.
It follows that, for all $0\le m\le Cn$, writing $t$ for $m/n$ we have 
\begin{align}
\bb{ \rho_1((m+1)/n)-\rho_1(m/n)} n&= \rho_1'(t)+ O(n^{-1}) \rlap{\hbox{\quad and}} \label{trendhyp2}\\
\bb{ e((m+1)/n)-e(m/n)} \omega n^{1/2}&= e'(t)\omega n^{-1/2}+ O(n^{-1}).\label{trendhyp3}
\end{align}
Let $A_m$ denote the event that at step $m$ the desired estimate
\begin{equation}\label{goodevent}
\rho_1(m/n)n - e(m/n)\omega n^{1/2}\le I_m\le \rho_1(m/n)n + e(m/n)\omega n^{1/2}
\end{equation} 
holds, and define the stopping time $T$ to be the first step $m\le Cn$ for
which $A_m$ fails, or $Cn$ if there is no such a step. Set $\hI_m^+=I_{m\wedge T}^+$ and $\hI_m^-=I_{m\wedge T}^-$. We first show that $(\hI_m^+)_{m=0}^{Cn}$ is a submartingale.
If $A_m$ does not hold, then $\hI_{m+1}^+=\hI_m^+$ (by definition of the stopping time),
and so $\E[\hI_{m+1}^+-\hI_m^+\mid\mathcal F_{m}]=0$.
Suppose that $A_m$ holds.
Then, putting \eqref{trendhyp1}--\eqref{trendhyp3} together, and writing $t$ for $m/n$ as before, we obtain
\begin{align}
&\E[\hI_{m+1}^+-\hI_m^+\mid\mathcal F_{m}]\nonumber\\
&= -2 \left(\frac{I_m}{n}\right)^2 -2\left(1-\left(\frac{I_m}{n}\right)^2\right)\frac{I_m}{n}
   - \rho_1'(t)+ e'(t)\omega n^{-1/2}+ O(n^{-1}).\label{trendhyp4}
\end{align}
From \eqref{goodevent}, recalling that $e(t)=\Theta(1)$ and $0\le \rho_1(t)\le 1$, we have
\begin{align}
&-2 \left(\frac{I_m}{n}\right)^2 -2\left(1-\left(\frac{I_m}{n}\right)^2\right)\frac{I_m}{n}\nonumber\\
&\ge -2\left(\rho_1(t) + e(t)\omega n^{-1/2}\right)^2 -2\left(\rho_1(t) + e(t)\omega n^{-1/2}\right)+2\left(\rho_1(t) - e(t)\omega n^{-1/2}\right)^3\nonumber\\
&=-2\rho_1(t)^2-2\rho_1(t)+2\rho_1(t)^3 -\left(4\rho_1(t)+2+6\rho_1(t)^2\right)e(t)\omega n^{-1/2}
     +O(\omega^2n^{-1}) \nonumber\\
&\ge \rho_1'(t) - 13 e(t)\omega n^{-1/2},\label{trendhyp5}
\end{align}
if $\omega\le n^{1/4}$, say, and $n$ is large enough. Here we used the fact that $\rho_1(t)$ solves \eqref{rho1}.
Substituting \eqref{trendhyp5} into \eqref{trendhyp4}, we have
\begin{align*}
\E[\hI_{m+1}^+-\hI_m^+\mid\mathcal F_{m}]
&\ge  \left(\rho_1'(t) - 13 e(t)\omega n^{-1/2}\right) - \rho_1'(t)+ e'(t)\omega n^{-1/2}+ O(n^{-1})\nonumber\\
&= - 13 e(t)\omega n^{-1/2} + 14e(t)\omega n^{-1/2}+ O(n^{-1})\ge 0.\label{trendhyp5}
\end{align*}
Thus  $(\hI_{m}^+)_{m=0}^{Cn}$ is a submartingale.
An almost identical argument, which we omit, shows that $(\hI_m^-)_{m=0}^{Cn}$ is a supermartingale.

In order to apply the Azuma--Hoeffding inequality we need to bound $\hI_{m+1}^+ - \hI_m^+$.
We always have $I_{m+1}-I_m\in \{0,-1,-2\}$.
Since $0\le -\rho_1'(t)\le 2$ for all $t$, and $e'(t)=O(1)$ for $t=O(1)$, if $n$ is large enough
then for $0\le t\le C$ we have
$|\rho_1'(t) \pm e'(t)\omega n^{-1/2}+ O(n^{-1})|\le 3$. 
Recalling the definition \eqref{Impdef} of $I_m^+$ and \eqref{trendhyp2}, \eqref{trendhyp3}  
it follows that
$|\hI_{m+1}^+ - \hI_m^+|\le |I_{m+1}^+ - I_m^+|\le 5$; 
a similar bound holds for $\hI_m^-$.
Since $\hI_0^+=\omega n^{1/2}$ and $\hI_0^-=-\omega n^{1/2}$, by (the supermartingale variant of)
the Azuma--Hoeffding inequality we have
\begin{align*}
&\Pr\bb{\exists 0\le m\le Cn\hbox{ s.t.\ }A_m\hbox{ fails}}\\
&\hspace{5ex}= \Pr\bb{\exists 0\le m\le Cn\hbox{ s.t.\ }\hI_m^+ <0 \  \vee \  \hI_m^- >0}\\
&\hspace{5ex}= \Pr\bb{\exists 0\le m\le Cn\hbox{ s.t.\ }\hI_0^+ - \hI_m^+ > \omega n^{1/2} \ \vee \ \hI_m^- - \hI_0^- > \omega n^{1/2}}\\
&\hspace{5ex}\le 2 \exp\left(-\frac{(\omega n^{1/2})^2}{50 Cn}\right)\to 0.
\end{align*}
Therefore, with probability $1-o(1)$ the event $A_m$ (the estimate \eqref{goodevent}) holds for all $0\le m \le Cn$. 
This implies \eqref{gooddef} with $e^{14C}\omega$ in place of $\omega$.
\end{proof}%

\begin{remark}\label{rem_precise}
It is easy to see that we cannot prove tighter bounds on the fraction $I_m/n$ of vertices
in isolated components than the $+\Op(n^{-1/2})$ bound above. This bound feeds in to
the error terms in many estimates throughout the paper, and is ultimately the reason
why we can only handle trees up to size around $n^{1/2}$, rather than $n^{2/3}$ as we
would ideally like. To push our methods further, it might be necessary to write $I_m/n$
as $\rho_1(m/n)+F(m/n)n^{-1/2}$ for some random function $F(t)$ of order $1$,
and to calculate taking this extra term into account. It is even possible that, writing
the formulae in the right way, the $F(\cdot)$ terms would drop out to first order,
though this is very far from clear.
\end{remark}

\subsection{A continuous-time idealised model}\label{ss_ct}
In the light of Lemma~\ref{iconc} and the formulae \eqref{pisol} and \eqref{pnisol}, if we consider
a not-too-large set $V$ of vertices, then we may expect the edges incident with $V$ to
evolve roughly according to the following continuous-time model.%

Firstly, in a small interval $[t,t+\dt]$ (corresponding to $n\dt$ steps
of the random graph process) a given edge $uv$ within $V$ appears with probability
\begin{equation}\label{contp1}
 2n^{-1}(2-\rho_1(t)^2) \dt \hbox{ \ or \ } 2n^{-1}(1-\rho_1(t)^2) \dt
\end{equation}
according to whether or not $u$ and $v$ are currently both isolated. Secondly, a given vertex $v\in V$ becomes
connected to some vertex outside $V$ with probability
\begin{equation}\label{contp2}
  2(1-\rho_1(t)^2 +\rho_1(t)) \dt \hbox{ \ or \ } 2(1-\rho_1(t)^2) \dt,
\end{equation}
according to whether $v$ is currently isolated or not. Here the second formula
comes from multiplying \eqref{pnisol} by the number $n\dt$ of steps and the number $\sim n$ of
vertices outside $V$, and the first from adding to this a term corresponding
to the difference between \eqref{pisol} and \eqref{pnisol} for each of the $\sim \rho_1(t)n$ isolated
vertices outside $V$.%

In what follows, we show that the discrete process is indeed well approximated by this continuous model, as
long as $|V|=o(\sqrt{n})$.%

\section{Ignoring bad events}
From now on we fix a constant $C$, and consider only $m\le Cn$. Let
$\good_m$ denote the event $|I_m/n-\rho_1(m/n)|\le \omega n^{-1/2}$, and $\good$ the event
$\bigcap_{m\le Cn} \good_m$, so $\good$ holds whp
by Lemma~\ref{iconc}. The next lemma
captures in a precise way the idea that we can `ignore what happens when $\good$ fails'.
We must be a little careful, as we shall calculate the probabilities of very unlikely
events $A$, and then sum over many events. So a bound of the form $\Pr(A)= \Pr(\tA)\pm \Pr(\good^\cc)$
is not enough, and we need the slightly fussy result below. From now on we
adopt the convention that if $(A_i)$ is a sequence of events, then $A_{\le i}$ denotes
the event $\bigcap_{j=1}^i A_j$.%

\begin{lemma}\label{lcond}
Let $(\Omega,\cF,\Pr)$ be a probability space without atoms. Let $(\cF_m)_{m\ge 0}$ be a filtration
of $\cF$, and for each $i$ let $\good_i$ and $A_i$ be $\cF_i$-measurable events.
Suppose that for all $i\le m$ we have
\begin{equation}\label{cpb}
 p_i \le \Pr( A_i \mid \cF_{i-1} ) \le q_i
\end{equation}
whenever $A_{\le i-1}\cap \good_{\le i-1}$ holds. Then there exist events $B_1,\ldots,B_m$ such that
\begin{equation}\label{ABH}
 A_{\le i} \symdiff B_{\le i}  \subseteq (\good_{\le i-1})^\cc
\end{equation}
for $i\le m$ and
\begin{equation}\label{est}
 \prod_{j=1}^i p_j \le \Pr(B_{\le i}) \le \prod_{j=1}^i q_j.
\end{equation}
\end{lemma}
In other words, in evaluating $\Pr(A_{\le m})$ we can `pretend' that the conditional probability bound \eqref{cpb}
holds whenever $A_{\le i-1}$ holds (in which case $\Pr(A_{\le m})$ would satisfy the inequalities
in \eqref{est}), except that we must allow ourselves to `cook' (i.e., modify) the event $A_{\le m}$ when $\good_{\le m}$ fails to hold.%
\begin{proof}
The idea is simply to `cook' the event whose probability we are estimating when (if ever) $\good_i$ first fails to hold.
It should be clear that this can be made to work, but let us spell out the details.%

First, replacing each $\cF_i$ by the $\sigma$-algebra generated by $A_1,\ldots,A_i$ and $\good_1,\ldots,\good_i$, we may
assume that each $\cF_i$ is finite. Since our probability space is atomless, we may construct
random variables $U_1,\ldots,U_m$ that are uniformly distributed
on $[0,1]$ and are independent of each other and of $\cF_m$.
Let $\cF_i^*$ denote the $\sigma$-algebra generated by $\cF_i$ and $U_1,\ldots,U_i$.
Set
\[
 B_i = \bb{\good_{\le i-1} \cap A_i} \cup \bb{ (\good_{\le i-1})^\cc \cap \{U_i\le p_i\}}.
\]
Then $B_i$ is $\cF_i^*$-measurable. Also, since $\good_{\le i-1}$ is $\cF_{i-1}^*$-measurable,
we have
 (i) $\Pr(B_i\mid \cF_{i-1}^*) = \Pr(A_i\mid \cF_{i-1}^*)$ on the event $\good_{\le i-1}$ and
(ii) $\Pr(B_i\mid \cF_{i-1}^*)=p_i$ on  $(\good_{\le i-1})^\cc$.%

By definition $B_i\symdiff A_i\subseteq (\good_{\le i-1})^\cc$, from which \eqref{ABH} follows. From (i)
and \eqref{cpb} it follows that when $\good_{\le i-1}\cap B_{\le i-1}=\good_{\le i-1}\cap A_{\le i-1}$ holds
we have $p_i\le \Pr(B_i\mid \cF_{i-1}^*)\le q_i$. Hence, by (ii), this inequality
holds whenever $B_{\le i-1}$ holds, and \eqref{est} follows by induction.
\end{proof}%

Lemma~\ref{lcond} effectively allows us to pretend that the event $\good$ defined in \eqref{gooddef}
always holds when applying the first and second moment methods. More precisely,
suppose that for each $n$ we have some list $A_1,\dots,A_{N(n)}$ of events, for example
the events that each possible tree with $k$ vertices is present as a component
in the $n$-vertex, $m$-edge graph $G_m=G_m^{(n)}$, and let $X_n$ be the number of these events
that hold. Let $\tA_1,\ldots,\tA_{N(n)}$ be the corresponding events
whose existence is guaranteed by the lemma, whose probabilities we can estimate using \eqref{est},
and let $\tX_n$ be the number of $\tA_i$ that hold.
If $\tmu_n=\E[\tX_n]\to 0$
then $\tX_n=0$ whp, and hence (since $\good$ holds whp, and $X_n=\tX_n$ when $\good$ holds)
$X_n=0$ whp. This works even if we cannot estimate $\mu_n=\E[X_n]$, or indeed if $\mu_n$ does not tend
to $0$. Similarly, if we can show that the expected number
of pairs of events $\tA_i$ that hold is asymptotically $\tmu_n^2$, then by the second
moment method we see that $\tX_n/\tmu_n\pto 1$, so $X_n/\tmu_n\pto 1$.%

\begin{remark}\label{RtE}
We shall apply Lemma~\ref{lcond} to certain basic events $A$ of the form
that certain edges appear in the random graph at certain times
and other edges have not appeared by a certain time. In each case we write
$\tA$ for an event satisfying the conditions \eqref{ABH} and \eqref{est}. (Thus
$\tA$ is not uniquely defined, but this will not matter.) If $X$ is a `counting random
variable' given by the sum of the indicator functions $1_{A_j}$ of some basic events $A_j$,
then we write $\tX$ for a corresponding `cooked' variable $\sum_j 1_{\tA_j}$.
If $E$ is a disjoint union of basic events $A_j$, then we will assume that the
cooked events $\tA_j$ are also disjoint, and write $\tE$ for their union.
It is not hard to check that the $\tA_j$ may be defined so that they are disjoint;
it is also not necessary: in what follows we can always replace $\Pr(\tE)$ by $\E \tX$
for suitable $X$. The only reason to consider $\tE$ is to make the notation more
concise and intuitive.
\end{remark}%

\section{Subgraph probabilities}%

\subsection{The basic estimate}\label{ss_basic}%
We now apply the ideas of the previous section to estimate the probability that a particular
`small' subgraph $H$ appears as a component of $G_m=G_m^{(n)}$, where $m\le Cn$. Although in the end
it is components that we wish to study, in what follows we allow $H$ to be disconnected
since this will be needed when considering moments of the number of components
of a certain size.%

Fix $n$ and $m$, and let $H$ be a graph with vertex set $V\subset [n]$,
with $k=|V|$ vertices and $\ell$ edges. When it comes to asymptotic
estimates, we shall allow $m$, $k$ and $\ell$ to grow with $n$,
but we always assume that
\[
 m=O(n) \hbox{ \ and \ } k=o(\sqrt n/\omega),
\]
where, as before, $\omega=\omega(n)\to\infty$ slowly.
In this subsection we further assume that $\ell=O(k)$.%

Let $E_H$ be the event that each component of $H$ is a component
of $G_m$, i.e., that $G_m$ contains all edges of $H$ but
no other edges incident with $V$.
Clearly, $E_H$ is the disjoint union of the events $E_{H,\prec,\bm}$,
where $\prec$ runs over all $\ell!$ orders on $E(H)$ and $\bm$
over all $\ell$-tuples $(m_1,\ldots,m_\ell)$ with $0<m_1<m_2<\cdots<m_\ell \le m$,
and $E_{H,\prec,\bm}$ is the event that $E_H$ holds,
with the edges of $H$ appearing during the process $(G_i)$ in the order $\prec$,
with the $j$th edge appearing at step $m_j$.%

For the moment, fix $\prec$ and $\bm$, and let $E=E_{H,\prec,\bm}$.
Let the edges of $H$, in the order $\prec$, be $f_1,\ldots,f_\ell$.
Recall that $e_i$
is the random edge added to $G_{i-1}$ to form $G_i$. Thus $E$ is
the event $A_{\le m}=A_1\cap\cdots\cap A_m$ where
\[
 A_{m_j} = \{ e_{m_j} = f_j \}
\]
for $j=1,\ldots,\ell$, and
\[
 A_i = \{ e_i \hbox{ joins two vertices outside }V \}
\]
for all $i\le m$ that are not equal to any $m_j$.%

As before, let $\good_i$ be the event that $|I_i/n-\rho_1(i/n)|\le \omega n^{-1/2}$  
and let $\good_{\le i-1}=\bigcap_{j\le i-1} \good_j$.
Whenever $\good_{\le i-1}$ holds, we have
$I_{i-1}/n=\rho_1((i-1)/n)+O(\omega n^{-1/2})=\rho_1(i/n)+O(\omega n^{-1/2})$.
Also, whenever $A_{\le i-1}$ holds, then in $G_{i-1}$ the edges incident with the vertices of
$V$ are precisely those edges $f_j\in E(H)$ for which $m_j<i$.
Hence, from \eqref{pisol} and \eqref{pnisol}, for each $j\le \ell$,
whenever $A_{\le m_j-1}\cap \good_{\le m_j-1}$ holds we have
\[
 \Pr(A_{m_j} \mid \cF_{m_j-1}) = 2n^{-2}(\alpha_j-\rho_1(m_j/n)^2)+O(\omega n^{-5/2}),
\]
where $\alpha_j$ (which depends on $H$ and on $\prec$) is equal to $2$ if none of
$f_1,\ldots,f_{j-1}$ shares a vertex with $f_j$, and is $1$ otherwise.
For later use, we write this bound as
\begin{equation}\label{cp1}
 \Pr(A_{m_j} \mid \cF_{m_j-1}) = n^{-2} g_j(m_j/n) +O(\omega n^{-5/2}),
\end{equation}
where
\begin{equation}\label{gj}
 g_j(t)=g_{H,\prec,j}(t)=2(\alpha_j-\rho_1(t)^2).
\end{equation}
Note that $g_j$ depends on $H$ and on the order $\prec$,
but only via the isomorphism type of the ordered graph $(H,\prec)$.%

Suppose that $i\le m$ is not one of the $m_j$. Then $m_j<i<m_{j+1}$ for some
$j=0,1,\ldots,\ell$, where we take $m_0=0$ and $m_{\ell+1}=m+1$.
Let $\beta_j$ be the number of vertices in $V$ that are not incident with any
of $f_1,\ldots,f_j$. Then whenever $A_{\le i-1}\cap \good_{\le i-1}$ holds,
in $G_{i-1}$ there are precisely $\beta_j$ isolated vertices in $V$
and $\rho_1(i/n)n+O(\omega\sqrt{n})+O(k)=\rho_1(i/n)n+O(\omega\sqrt{n})$
isolated vertices outside $V$.
Since there are in total $k(n-k)=kn+O(k^2)$ possible edges between $V$ and $V^\cc$
and $\binom{k}{2}-j=O(k^2)$ possible edges inside $V$ that could be added at step $i$,
in total there are $kn+O(k^2)$ possible edges whose selection as $e_i$ would 
mean that $A_i$ does not hold, and of these $\beta_j\rho_1(i/n) n+O(k\omega\sqrt{n})$ join two isolated vertices.
It follows from \eqref{pisol} and \eqref{pnisol} that when $A_{\le i-1}\cap \good_{\le i-1}$ holds
and $m_j<i<m_{j+1}$, then writing $\rho_1$ for $\rho_1(i/n)$, we have
\begin{eqnarray*}
 \Pr(A_i^\cc \mid \cF_{i-1}) &=& (kn+O(k^2)) \bb{2n^{-2}(1-\rho_1^2)+O(\omega n^{-5/2})} \\
 &&\quad+\quad (\beta_j\rho_1 n+O(k\omega\sqrt{n})) \bb{2n^{-2}+O(\omega n^{-5/2})}.
\end{eqnarray*}
Hence
\begin{eqnarray*}
 \Pr(A_i^\cc \mid \cF_{i-1})
 &=& 2kn^{-1}(1-\rho_1^2)+2\beta_j\rho_1 n^{-1} + O(k^2n^{-2}+k\omega n^{-3/2}) \\
 &=& \bb{2k(1-\rho_1^2)+2\beta_j\rho_1}n^{-1} +O(k\omega n^{-3/2}) \\
 &=& h_j(i/n) n^{-1} +O(k\omega n^{-3/2}),
\end{eqnarray*}
where
\begin{equation}\label{hj}
 h_j(t) = h_{H,\prec,j}(t) = 2k(1-\rho_1(t)^2)+2\beta_j\rho_1(t).
\end{equation}
Since this probability is of order $kn^{-1}$, using $\log(1-x)=-x-O(x^2)$ for $x\le 1/2$,
we see that under these assumptions,
\begin{equation}\label{cp2}
 \log \Pr(A_i\mid \cF_{i-1})  = -h_j(i/n) n^{-1}+O(k\omega n^{-3/2}).
\end{equation}%

The conditional probability estimates \eqref{cp1} and \eqref{cp2} are valid
only when $A_{\le i-1}\cap \good_{\le i-1}$ holds. But by Lemma~\ref{lcond}, this
is enough to deduce that there is a `cooked' event $\tE=\tE_{H,\prec,\bm}$
that agrees with $E=E_{H,\prec,\bm}$ off $\good=\good_{\le m}$, such that
\begin{multline*}
 \Pr(\tE) = \prod_{j=1}^\ell \bb{n^{-2}g_j(m_j/n)+O(\omega n^{-5/2})} \\
   \exp\left(-\sum_{j=0}^\ell \sum_{m_j<i<m_{j+1}} \bb{h_j(i/n)n^{-1}+O(k\omega n^{-3/2})}\right).
\end{multline*}
From now on we adopt the rather ugly convention that%
\begin{equation}\label{appconv}
 x(n) \approx y(n)\hbox{\quad means\quad}x(n)=y(n)\exp\bb{O(k\omega n^{-1/2})},
\end{equation}
noting that $x\approx y$ implies $x\sim y$ by \eqref{kosn}. On a first reading, the reader
may wish to read $\approx$ simply as $\sim$; the more precise error estimate will be 
relevant only in Section~\ref{sec_supcrit}. With this convention,
since the term $O(k\omega n^{-3/2})$ appears $m-\ell\le m=O(n)$ times in the formula above,
we have
\[
 \Pr(\tE) \approx n^{-2\ell} \prod_{j=1}^\ell \bb{g_j(m_j/n)+O(\omega n^{-1/2})}\ \exp\left(-n^{-1}\sum_{j=0}^\ell \sum_{m_j<i<m_{j+1}} h_j(i/n)\right).
\]%

Let us remark that the calculations that follow would be much simpler if the error terms were
all multiplicative, i.e., we could replace $g_j(t)+O(\omega n^{-1/2})$ by $g_j(t)(1+O(\omega n^{-1/2}))$. Unfortunately,
this is not the case when $g_j$ corresponds to joining non-isolated vertices and $t$
is small: then $g_j(t)$ is of order $1-\rho_1(t)^2$, which is small for $t$ near $0$.
Let us write $t_j$ for $m_j/n$, $1\le j\le \ell$, and $t_0=0$ and $t_{\ell+1}=t=m/n$.%

Noting that the functions $g_j$ and $h_j$ do not depend on $n$ (only on the isomorphism
type of $(H,\prec)$) and are smooth,
and in particular that $h_j$ and its derivatives are $O(k)$,  it is easy to see that
\[
 n^{-1}\sum_{m_j<i<m_{j+1}} h_j(i/n) = \int_{t_j}^{t_{j+1}} h_j(x) \dx +O(k/n).
\]
Since $k\ell/n\le k/\sqrt{n} \le k\omega/\sqrt{n}=o(1)$, writing $\eta$ for $\omega n^{-1/2}$, it follows that
\begin{equation}\label{Eb1}
 \Pr(\tE_{H,\prec,\bm}) \approx n^{-2\ell} \prod_{j=1}^\ell \bb{g_j(t_j)+O(\eta)}\ \exp\left(-\sum_{j=0}^\ell \int_{t_j}^{t_{j+1}} h_j(x) \dx \right).
\end{equation}
Our next aim is to sum this expression over all $\bm=(m_1,m_2,\ldots,m_\ell)$ with $0<m_1<\cdots<m_\ell\le m$, 
and then approximate the sum by an integral. This approximation step could be circumvented by passing
to continuous time from the beginning, at the cost of (mostly notational) complications elsewhere. Still,
we spell it out in some detail, even though it is not at all surprising.%

Write $F(t_1,\ldots,t_\ell)=F_{H,\prec}(t_1,\ldots,t_\ell)$ for the term inside the exponential
in \eqref{Eb1}, without the minus sign. Note from \eqref{hj} that for $1\le j\le \ell$ we have
\begin{equation}\label{delF}
 \frac{\partial F}{\partial t_j} = h_{j-1}(t_j)-h_j(t_j) = 2(\beta_{j-1}-\beta_j)\rho_1(t_j) = O(1),
\end{equation}
since adding an edge can destroy at most two isolated vertices, so $|\beta_j-\beta_{j-1}|\le 2$. 
We claim that the sum
\[
 S = n^{-2\ell} \sum_{0<m_1<\cdots<m_\ell\le m} 
    \prod_{j=1}^\ell \bb{g_{H,\prec,j}(t_j)+O(\eta)}\exp(-F_{H,\prec}(t_1,\ldots,t_\ell)),
\]
with $t_j=m_j/n$, can be bounded within a factor $\exp(O(k\omega/\sqrt{n}))$ by the integral
\[
 I =  n^{-\ell} \int_{0<t_1<\cdots<t_\ell\le t} \prod_{j=1}^\ell \bb{g_{H,\prec,j}(t_j)+O(\eta)}\exp(-F_{H,\prec}(t_1,\ldots,t_\ell)).
\]
More precisely, this means that for all $C$ there exists $C'$ such that taking the implicit constant in the $O(\eta)$
terms in $S$ to be $C$, any quantity satisfying this upper/lower bound is bounded by $\exp(O(k\omega/\sqrt{n}))$
times something satisfying the bound given for $I$ with $C'$ as the implicit constant. To see this, note that
we can think of both expressions as of the form
\[
 \int_\cS  \   \prod_{j=1}^\ell \bb{g_{H,\prec,j}(t_j)+O(\eta)}\exp(-F_{H,\prec}(t_1,\ldots,t_\ell)) \,\dnu(t_1,\ldots,t_\ell),
\]
where $\cS$ is the simplex $\{(t_1,\ldots,t_\ell): 0<t_1<\cdots<t_\ell\le t\}\subset [0,t]^\ell$ and $\nu$ is an appropriate measure
on $\cS$: for $I$ we just take $\nu=\nu_I$ to be $n^{-\ell}$ times Lebesgue measure; for $S$ we take $\nu_S$
to be the discrete measure assigning mass $n^{-2\ell}$
to each element of the set $\cS_0$ of sequences $0<t_1<\cdots<t_\ell\le t$ in which every $t_i$ is a multiple of $1/n$.
To show that $S$ and $I$ are close, we show that the measures are close, if we allow ourselves to `shift' the points $t_i$ slightly.
First note that shifting every $t_i$ by an amount that is $O(\ell/n)=O(k/n)$ does not affect the integrand significantly. Indeed,
since each $g_j$ is $O(1)$-Lipschitz, it changes by $O(k/n)=o(\eta)$, an amount that is easily absorbed into the $+O(\eta)$
error term. Secondly, from \eqref{delF}, the total effect on $F_{H,\prec}$ is to change it by order
$O(\ell k/n)=O(k^2/n)=O(k\omega/\sqrt{n})$, giving a multiplicative error within the bound we are aiming for.
To compare the measures, first consider starting with $(t_1,\ldots,t_{\ell})\in \cS$ and rounding each $t_i$ up to the
nearest multiple of $1/n$. Every $(t_1',\ldots,t_\ell')\in \cS_0$ has preimage a cube of volume $n^{-\ell}$;
since $|t_i-t_i'|\le 1/n=O(\ell/n)$ this gives the required upper bound on $S$ of the form $\approx I$.
(We do not get a lower bound this way since not all of $\cS$ maps to $\cS_0$: we may have $t_i'=t_{i+1}'$.)
For the reverse bound, consider a different map from $\cS$ to $\cS_0$: first add $2j/n$ to each $t_j$, then
rescale all $t_j$ by $(t-2\ell/n)/t=1-O(\ell/n)$, then round up to the nearest multiple $t_j'$ of $1/n$. It is 
easy to see that this maps $\cS$ into $\cS_0$, that $|t_j-t_j'|=O(\ell/n)$ for all $j$, and that the preimage
of any point of $\cS_0$ has volume at most $n^{-\ell}(1-O(\ell/n))^{-\ell} = n^{-\ell}(1+O(\ell^2/n)) \approx n^{-\ell}$,
recalling that $\ell=O(k)=o(\sqrt{n})$ and so $\ell^2/n = O(k/\sqrt{n})=O(k\omega/\sqrt{n})$. This gives an upper
bound on $I$ in terms of $S$, and hence a lower bound on $S$ in terms of $I$,
that combined with the bound above, completes the approximation.
Summarising, from \eqref{Eb1} and the approximation above we have%
\begin{multline*}
\sum_{0<m_1<\cdots<m_\ell\le m} \Pr(\tE_{H,\prec,\bm}) \\
 \approx n^{-\ell} \int_{0<t_1<\cdots<t_\ell\le t} \prod_{j=1}^\ell \bb{g_{H,\prec,j}(t_j)+O(\eta)}\exp(-F_{H,\prec}(t_1,\ldots,t_\ell)).
\end{multline*}%

Recall that $E_H$ (the event that the components of $H$ are present as components of $G_m$)
is the disjoint union of the events $E_{H,\prec,\bm}$. Following the convention
described in Remark~\ref{RtE},
summing over orders $\prec$ we obtain the formula
\begin{equation}\label{PtE}
 \Pr(\tE_H) = \sum_\prec \sum_{\bm} \Pr(\tE_{H,\prec,\bm}) \approx n^{-\ell} \mu_H
\end{equation}
where $\mu_H=\mu_{H,t,n}$ is given by
\begin{equation}\label{muH}
 \mu_H = \sum_\prec  \int_{0<t_1<\cdots<t_\ell<t} \prod_{j=1}^\ell \bb{g_{H,\prec,j}(t_j)+O(\eta)}\exp(-F_{H,\prec}(t_1,\ldots,t_\ell)).
\end{equation}%

We are perhaps abusing notation here; $\mu_H$ does not denote a single quantity, but rather a range -- we use it
simply as short-hand for the formula above. Let $B$ be the constant implicit in the $O(\eta)$ notation,
and define $\mu_H^+$ by taking $B\eta$ in place of each $O(\eta)$ term in \eqref{muH}. Similarly,
define the minimum value $\mu_H^-$ by taking $-B\eta$ or, if $g_{H,\prec,j}(t_j)<B\eta$ for some $j$, setting $\mu_H^-=0$.
Then any occurrence of $\mu_H$ stands for some quantity in the range $[\mu_H^-,\mu_H^+]$,
with different occurrences perhaps being different. (This is the usual behaviour of $O(\cdot)$ notation.)
Occasionally, when giving upper bounds, we shall be more explicit and work with $\mu_H^+$.%

Set
\begin{equation}\label{muH0}
 \mu_H^0 = \mu_{H,t}^0= \sum_\prec  \int_{0<t_1<\cdots<t_\ell<t} \prod_{j=1}^\ell g_{H,\prec,j}(t_j)\exp(-F_{H,\prec}(t_1,\ldots,t_\ell)),
\end{equation}
a quantity that depends only on $t$ and on the isomorphism type of $H$,
in contrast to $\mu_H^\pm$, which also depend on $n$ via $\eta$.
Since $\eta\to 0$ as $n\to\infty$, for $H$ and $t$ fixed we have
\begin{equation}\label{muHconv}
 \mu_H\to \mu_H^0 \hbox{ \ as \ }n\to\infty.
\end{equation}%

Comparing the formulae \eqref{gj} and \eqref{hj} to \eqref{contp1} and \eqref{contp2}, or simply noting
that the arguments in this section bound the deviation of the behaviour of the edges of $G_m$
incident with $V$ from the idealised description given in Subsection~\ref{ss_ct},
we see that $\mu_H^0$ is exactly the coefficient of $n^{-\ell}$ (the leading term) in the
probability that in the continuous-time model, $E_H$ holds at time $t$.%

Let $\cT_k$ denote the set of all $k^{k-2}$ trees on $[k]$, and define
\begin{equation}\label{muk}
 \mu_k = \mu_{k,t,n} = \frac{1}{k!} \sum_{T\in \cT_k} \mu_T \hbox{ \ and \ } \mu_k^0 = \mu_{k,t}^0 = \frac{1}{k!}\sum_{T\in \cT_k} \mu_T^0.
\end{equation}
For any $k=o(\sqrt n/\omega)$, by \eqref{PtE} the number $T_k$ of $k$-vertex tree components of $G_m$ satisfies
\begin{equation}\label{etT}
 \E \tT_k = \binom{n}{k} \sum_{T\in \cT_k} \Pr(\tE_T) \approx \frac{n^k}{k!} \sum_{T\in \cT_k} n^{-(k-1)}\mu_T = n\mu_k,
\end{equation}
where $\tT_k$ (the sum of the indicator functions of the relevant events $\tE_{T,\prec,\bm}$) is
a random variable that agrees with $T_k$ on the (very likely) event $\good$.%

From \eqref{muHconv}, it follows that for $k$ and $t$ fixed we have
\footnote{We are `cheating' slightly here: we know that $T_k=\tT_k$ with probability $1-o(1)$ (in particular,
on the event $\good$), but this does not itself imply that $\E T_k\sim \E \tT_k$. However, greatly simplified
versions of the arguments in Subsections~\ref{ss_basic} and \ref{ss_moments} show that for $k$ fixed, $\E \tT_k^2=O(n^2)$.
(For this we need only \eqref{muHconv} rather than Lemma~\ref{mkclose}, so there is no circularity.)
Since $T_k\le n$, so $\E(T_k^2)=O(n^2)$, it follows by Cauchy--Schwarz that $\E |T_k-\tT_k| =o(n)$,
so indeed $\E T_k\sim \E \tT_k$.}
\[
 \E T_k \sim \E \tT_k \sim n\mu_{k,t}^0.
\]
Since the probability that a given edge is added at a given step is $O(n^{-2})$,
the probability that a given set of $k=O(1)$ vertices forms a non-tree component in $G_m$
is $O(m^kn^{-2k})=O(n^{-k})$. It follows that for $k,t$ fixed,
\[
 \E N_k(G_m) = k \E T_k +O(1) \sim  k \E T_k \sim nk\mu_{k,t}^0.
\]
From \eqref{conv} we have $\E N_k(G_m)\sim n\rho_k(t)$, so
\begin{equation}\label{ident2}
 \mu_{k,t}^0 = \frac{\rho_k(t)}{k}
\end{equation}
for all $k\ge 1$ and $t>0$.%

It remains to handle the error terms in $\mu_H$ when $|H|$ grows with $n$, to deal with non-tree
components, and to perform a second moment estimate.%

\subsection{Non-tree components}\label{ss_nontree}
In the Erd\H os--R\'enyi model with $p=\Theta(1/n)$, if $k=o(\sqrt{n})$ then
it is easy to see that the expected number of $k$-vertex non-tree components
is much smaller than the expected number of $k$-vertex tree components:
the key observation is that any non-tree is obtained by adding some number
$r>0$ of edges to a tree, that there are at most $\binom{k}{2}=o(n)$ choices for each
added edge, and that each edge reduces the probability by a factor $p$.
In this subsection we formalize a similar argument for $G_m$. The problem
is that we do not have independence. Nevertheless, it is easy to see that
each extra edge reduces our upper bound on the probability of a certain subgraph
by a sufficient factor.%

Let $k=k(n)=o(\sqrt{n}/\omega)$ and let $1\le r=r(n)\le \binom{k}{2}-k+1$.
Let $\cC_{k,r}$ be the set of connected graphs on $[k]$ with $(k-1)+r$ edges,
and let $C_{k,r}$ be the number of components of $G_m$ that are isomorphic
to graphs in $\cC_{k,r}$, i.e., have $k$ vertices and $r$ excess edges.
Breaking down the event $E_H$ that a particular $H\in \cC_{k,r}$ is a component
of $G_m$ as before, we have
\[
 \E\tC_{k,r} = \binom{n}{k}\sum_{H\in \cC_{k,r}} \Pr(\tE_H)
  = \binom{n}{k} \sum_{H\in \cC_{k,r}} \sum_{\prec} \sum_{\bm}  \Pr(\tE_{H,\prec,\bm}).
\]
Given $H\in \cC_{k,r}$ and an order $\prec$ on $E(H)$, let $T\subset H$ be the spanning tree of $H$
formed by adding the edges of $H$ one-by-one in order, including only those edges that join
different components of the current graph. Let $R=E(H)\setminus E(T)$ be the
set of `redundant' edges, and let $\prec'$ be the order on $E(T)$ induced by $\prec$.
Then
\[
 \E\tC_{k,r} = \binom{n}{k} \sum_{T\in \cT_{k}} \sum_{\prec'} \sum_{\bm'} \sum_R \sum_\prec \sum_{\bm} \Pr(\tE_{H,\prec,\bm}),
\]
where $\prec'$ runs over all orders on $E(T)$, $\bm'$ over all $(k-1)$-tuples $0<m_1'<\cdots<m_{k-1}'\le m$,
$R$ runs over all sets of $r$ edges of $T^\cc$, $\prec$ over orders on $E(T)\cup R$
extending $\prec'$ and such that with $H=T\cup R$, $R$ is indeed the set of redundant edges,
and $\bm$ over all $E(H)$-tuples  $\bm$ that are compatible with $\bm'$. In other words, we choose in what order and
when the edges of $T$ appear, then we choose which redundant edges will appear, and when.%

As before, we have an estimate for each $\Pr(\tE_{H,\prec,\bm})$ as a product of factors for each time step,
of the form \eqref{cp1} or \eqref{cp2} for steps at which an edge of $H$ does or does not appear.
Let $P_{H,\prec,\bm}$ denote this estimate (or, to be concrete, the upper bound).  
We compare $P_{H,\prec,\bm}$ with the corresponding estimate $P_{T,\prec',\bm'}$
for $\Pr(\tE_{T,\prec',\bm'})$.
Because the edges of $R$ are redundant, at any given stage the edges of $H$ which have appeared
induce the same component structure on $[k]$ as those of $T$ which have appeared by the same
point. Thus we have the same estimate \eqref{cp1} or \eqref{cp2} for every step except those
at which the redundant edges appear. For these, in the case of $H$ we have some
probability of order $O(n^{-2})$, say at most $Cn^{-2}$, replacing a probability in
the case of $T$ that is $1-o(1)$ (see \eqref{cp2}), and is hence (for $n$ large) 
at least $1/2$.
Since there are at most $k^{2r}$ choices for the edges of $R$, and $m^r$ choices
for the steps at which they appear (which then specifies the order), we see 
that%
\footnote{Note that we cannot directly compare the expectations, only our bounds: in all steps the conditional
probabilities differ slightly, but except for those where redundant edges are added, the same
bound applies in both cases.}
\[
 \E\tC_{k,r} \le 
  k^{2r}m^r(2C/n^2)^r \binom{n}{k} \sum_{T\in \cT_{k}} \sum_{\prec'} \sum_{\bm'} P_{T,\prec',\bm'}.
\]
Since the argument leading to \eqref{PtE} bounded the estimates $P$, it follows from \eqref{PtE} that
\begin{eqnarray*}
 \E\tC_{k,r} 
 &\le&  (1+o(1))k^{2r}m^r(2C/n^2)^r \binom{n}{k} \sum_{T\in \cT_k} n^{-(k-1)} \mu_T^+ \\
  &\le& (1+o(1))(2Ck^2mn^{-2})^r n\mu_k^+,
\end{eqnarray*}
where, in analogy with \eqref{muk}, we set $\mu_k^+=\frac{1}{k!}\sum_{T\in \cT_k}\mu_T^+$.%

Finally, since $k^2mn^{-2}=O(k^2/n)=o(1)$, we conclude that
\begin{equation}\label{non-tree}
 \E\sum_{r\ge 1}\tC_{k,r} =O\left( \frac{k^2}{n} n\mu_k^+\right) = o(n\mu_k^+).
\end{equation}%

\subsection{Refined tree counts}\label{ss_refined}
To deal with the additive nature of the $+O(\eta)$ error terms in \eqref{muH} we need to understand
the behaviour of $\mu_{H,t}^0$ slightly better. Fix a graph $H$ on $[k]$ with $\ell$ edges,
and enumerate the edges of $H$ arbitrarily as $e_1,e_2,\ldots,e_\ell$. Thinking of $t_i$ as the `time'
at which edge $e_i$ appears in the continuous-time model of Subsection~\ref{ss_ct}, then using
this model, or arguing directly from \eqref{muH0}, we have the alternative formula
\begin{equation}\label{muH02}
 \mu_H^0 = \mu_{H,t}^0 = \int_{t_1=0}^t\cdots \int_{t_\ell=0}^t \prod_{j=1}^\ell \hg_{H,\prec,j}(t_j) e^{-\hF_H(t_1,\ldots,t_\ell)},
\end{equation}
where $\prec$ is the order on $E(H)$ corresponding to $t_1,\ldots,t_\ell$,
$\hg_{H,\prec,j}(t)=2(\alpha-\rho_1(t)^2)$ with $\alpha=2$ if in the subgraph of $H$ formed by
edges arriving before $e_j$ in the order $\prec$ both ends of $e_j$ are isolated,
and $\alpha=1$ otherwise, and
\begin{equation}\label{hF}
 \hF_H(t_1,\ldots,t_\ell) = 2k\int_{t'=0}^t (1-\rho_1(t')^2) + 2\sum_{v\in [k]} \int_{t'=0}^{t_v} \rho_1(t'),
\end{equation}
where $t_v=\min\{t_i: e_i \hbox{ meets }v\}$ is the time at which $v$ stops being isolated,
with $t_v=t$ if $v$ is isolated in $H$. Similarly,
from \eqref{muH} we have that
\begin{equation}\label{muH2}
 \mu_H = \mu_{H,t,n} = \int_{t_1=0}^t\cdots \int_{t_\ell=0}^t \prod_{j=1}^\ell
   \bb{\hg_{H,\prec,j}(t_j)+O(\eta)} e^{-\hF_H(t_1,\ldots,t_\ell)},
\end{equation}
where, as before, $\eta=\omega n^{-1/2}$.

From the description of the continuous-time model, or directly from \eqref{muH02},
it is easy to see that $\mu_H^0$ is multiplicative: if $H_1$ and $H_2$ are vertex-disjoint, then
\[
 \mu_{H_1\cup H_2,t}^0 = \mu_{H_1,t}^0\, \mu_{H_2,t}^0.
\]
Moreover, from \eqref{muH2}, the same is true of $\mu_H$ or, more concretely, of the upper
bound $\mu_H^+$ where we replace $O(\eta)$ by $B\eta$: for vertex-disjoint $H_1$ and $H_2$ we have
\begin{equation}\label{muHpmult}
  \mu_{H_1\cup H_2,t,n}^+ = \mu_{H_1,t,n}^+\, \mu_{H_2,t,n}^+.
\end{equation}%

Consider the contribution to \eqref{muH2} arising from taking the $O(\eta)$ term in the
final term in the product. This is very similar to $tO(\eta)=O(\eta)$ times
the formula for $\mu_{H-e_\ell}$. In fact, there are only two differences:%

\smallskip
(i) each factor $\hg_{H,\prec,j}(t_j)$, $j\ne \ell$, is less than or equal to
the corresponding factor $\hg_{H-e_\ell,\prec',j}(t_j)$, where $\prec'$ is $\prec$ restricted
to $e_1,\ldots,e_{\ell-1}$. The reason is that if the ends of $e_j$ are isolated
in the relevant subgraph of $H$, they are certainly isolated in the corresponding subgraph
of $H-e_\ell$, obtained by deleting $e_\ell$ if it is present, and the factor corresponding 
to the isolated vertices case is larger.\footnote{This is the only
part of the argument that is specific to the Bohman--Frieze rule. Elsewhere, everything
adapts \emph{mutatis mutandis} to general bounded-size rules. Here we can handle bounded-size rules
with a suitable monotonicity property with essentially no change to the argument.
It should be possible to handle the general
bounded-size case with a weaker error term, in the end treating components up to size $n^c$
for some $0<c<1/2$ that depends on the rule; the problem is simply bounding the additive
error term, and a solution might be to show that the contribution from terms where $g(\cdot)$ is at most
some suitable negative power of $k$ is negligible.}%

(ii) The quantities $\hF_H(t_1,\ldots,t_\ell)$ and $\hF_{H-e_\ell}(t_1,\ldots,t_{\ell-1})$ differ,
but by at most $O(1)$. This is because the presence or absence of $e_\ell$ only affects
whether two vertices (its ends) are isolated, and so only affects two terms in the sum in~\eqref{hF}.%

\smallskip
It follows that, overall, this term in the expansion contributes at most $O(\eta)\mu_{H-e_\ell}$ to $\mu_H$.
A similar argument applies to all other terms in the product, and we conclude that
\begin{equation}\label{mudiff}
 |\mu_H -\mu_H^0| \le \sum_{e\in E(H)} C\eta\mu_{H-e}^+,
\end{equation}
for some constant $C$, where $\mu_{H-e}^+$ is the `worst case' instance of
the formula \eqref{muH} for $H-e$, obtained by replacing all
$+O(\eta)$ error terms  by $+B\eta$ for an appropriate
constant $B$. Using this bound and the asymptotics of $\rho_k(t)$ we can show
that for trees, at least in total, $\mu_T$ is close to $\mu_T^0$.%

Recall from \eqref{ident2} that $\mu_{k,t}^0=\rho_k(t)/k$. Thus,
by Theorem~\ref{RW}, for $k\ge 1$ and $t\in I_0$, where $I_0$ is as in \eqref{I0def}, we have
\begin{equation*}
 \mu_{k,t}^0 = \gamma(t) k^{-5/2} e^{-\delta(t) k}(1+O(1/k)),
\end{equation*}
where $\gamma(t)$ is bounded and bounded away from $0$.
In particular, for $t\in I_0$ and any $k\ge 1$ we have
\begin{equation}\label{ktheta}
 \mu_{k,t}^0 = \Theta( k^{-5/2} e^{-\delta(t) k} ).
\end{equation}
\begin{lemma}\label{mkclose}
For any $k=k(n)=o(\sqrt{n}/\omega)$ and $t=t(n)\in I_0$ we have
\[
 \mu_k = (1+O(k\omega/\sqrt{n}))\mu_{k,t}^0 \sim\mu_{k,t}^0,
\]
where $\omega=\omega(n)\to\infty$ is the quantity appearing in \eqref{gooddef}, and $I_0$ is as in \eqref{I0def}.
\end{lemma}
\begin{proof}
Recall that $\mu_k$ depends on $t$ and (via $\eta$) on $n$, while $\mu_{k,t}^0$ is independent of $n$.
From \eqref{mudiff} and the multiplicativity \eqref{muHpmult} of $\mu^+$ we have
\[
 |\mu_k-\mu_k^0| \le \frac{C\eta}{k!} \sum_{T\in \cT_k}\sum_{e\in E(T)} \mu_{T_1}^+\mu_{T_2}^+,
\]
where $T_1$ and $T_2$ are the two components of $T-e$. Since two given trees $T_1$ and $T_2$
on complementary subsets of $[k]$ of size $k_1$ and $k_2$ arise from exactly $k_1k_2$ pairs $(T,e)$,
the bound above is exactly
\begin{equation}\label{eq12}
 \frac{C\eta}{2k!}\sum_{k_1+k_2=k} k_1k_2 \binom{k}{k_1} \sum_{T_1\in \cT_{k_1}} \mu_{T_1}^+ \sum_{T_2\in \cT_{k_2}} \mu_{T_2}^+ = \frac{C\eta}{2} \sum_{k_1+k_2=k} k_1k_2\mu_{k_1}^+\mu_{k_2}^+.
\end{equation}
Let $\kmax(n)=o(\sqrt{n}/\omega)$ and $t=t(n)\in I_0$.
We claim that if $n$ is large enough, then for all $1\le k\le \kmax$ we have $\mu_{k,t,n}^+\le 2\mu_{k,t}^0$, say.
The proof is by induction on $k$, with the base case $k=1$ being trivial. For the induction step,
since $k_1,k_2<k$, the induction hypothesis, \eqref{eq12} and \eqref{ktheta} give
\begin{eqnarray}
 |\mu_{k,t,n}-\mu_{k,t}^0| &\le& 2C\eta \sum_{k_1+k_2=k} k_1k_2\mu_{k_1,t}^0 \mu_{k_2,t}^0 \nonumber \\
&=& O(\eta) \sum_{k_1+k_2=k} k_1^{-3/2}k_2^{-3/2} e^{-\delta(t) k} \label{eq13} \\
 &=& O(\eta k^{-3/2} e^{-\delta(t) k}), \nonumber
\end{eqnarray}
since the outer terms dominate the sum up to constant factors.
Since $k\eta\le \kmax\eta=o(1)$, using \eqref{ktheta} again this bound is
$O(k\eta\mu_{k,t}^0)=o(\mu_{k,t}^0)$, and certainly at most $\mu_{k,t}^0$ if $n$
is large enough. Since the bound applies to $\mu_{k,t,n}^+$ this completes the induction proof.
Applying \eqref{eq13} a final time, we obtain that
\[
 |\mu_{k,t,n}-\mu_{k,t}^0|=O(k\eta \mu_{k,t}^0) = O(k\omega n^{-1/2} \mu_{k,t}^0),
\]
completing the proof of the lemma.
\end{proof}%

Combining the results so far, we have established the asymptotics of the number
of $k$-vertex components for all $k=o(\sqrt{n}/\omega)$. Recall the notational convention \eqref{appconv}.
\begin{lemma}\label{lsofar}
Let $k=k(n)=o(\sqrt{n}/\omega)$ and $m=m(n)\in nI_0$. Setting $t=t(n)=m/n$ we have
\[
 \E \tT_k \approx n\mu_k \approx n\mu_{k,t}^0 = \frac{n}{k}\rho_k(t),
\]
and
\[
 \E \tC_k = O(k\rho_k(t)) = o(n\rho_k(t)/k),
\]
where $C_k$ is the number of non-tree components of $G_m^{(n)}$
of size $k$, and $\tC_k$ is the corresponding `cooked' random variable.
\end{lemma}
\begin{proof}
The first statement follows from \eqref{etT},
Lemma~\ref{mkclose} and \eqref{ident2}. The second follows from \eqref{non-tree},
noting that $\tC_k=\sum_{r\ge 1} \tC_{k,r}$,
and recalling that Lemma~\ref{mkclose} applies with $\mu_k^+$ in place of $\mu_k$,
to give $\mu_k^+\sim \mu_k$.
\end{proof}%

\subsection{Higher moments}\label{ss_moments}
Let $T_{k_1,k_2}$ denote the number of ordered pairs $(T_1,T_2)$ of distinct tree components of $G_m$
where $T_i$ has $k_i$ vertices; here, as usual, $k_i=k_i(n)=o(\sqrt{n}/\omega)$. Writing $T_{k_1,k_2}$
as a sum of indicator functions of the corresponding events $\tE_{T_1\cup T_2}$, and defining
$\tT_{k_1,k_2}$ as the sum of the indicator functions of the corresponding `cooked' events
$\tE_{T_1\cup T_2}$, we see from \eqref{PtE} that\footnote{Recall that our notation $\approx$
hides errors of order $k\omega/\sqrt{n}$. Here $k=k_1+k_2$.}
\begin{eqnarray*}
 \E \tT_{k_1,k_2} &=& \binom{n}{k_1} \binom{n-k_1}{k_2}
      \sum_{T_1\in \cT_{k_1}}  \sum_{T_2\in \cT_{k_2}} \Pr( \tE_{T_1\cup T_2} ) \\
  &\approx& \frac{n^{k_1+k_2}}{k_1! k_2!} \sum_{T_1\in \cT_{k_1}}  \sum_{T_2\in \cT_{k_2}} n^{-(k_1-1+k_2-1)} \mu_{T_1\cup T_2},
\end{eqnarray*}
where $\mu_H$ is a quantity satisfying the bound in \eqref{muH}, or, equivalently, \eqref{muH2}.
Since the formula \eqref{muH2} is multiplicative over disjoint unions, we thus have
\begin{equation}\label{EtT12}
 \E \tT_{k_1,k_2} \approx  \frac{n}{k_1!} \sum_{T_1\in \cT_{k_1}} \mu_{T_1}
      \frac{n}{k_2!} \sum_{T_2\in \cT_{k_2}} \mu_{T_2}  = n^2 \mu_{k_1}\mu_{k_2} \approx n^2 \mu_{k_1}^0\mu_{k_2}^0,
\end{equation}
where the last step is from Lemma~\ref{mkclose}.
Let $K_1<K_2=o(\sqrt{n}/\omega)$ and set
\[
  X=\sum_{k=K_1}^{K_2-1} T_k(G_m) \hbox{ \ and \ }  \tX=\sum_{k=K_1}^{K_2-1} \tT_k(G_m),
\]
so $X=\tX$ with probability $1-o(1)$. Expressing $\tX$ as a sum of the indicator functions of the
(cooked) events that particular trees are present, by Lemma~\ref{lsofar} we have
\[
 \E\tX = \sum_{k=K_1}^{K_2-1} \E \tT_k \sim n\sum_{k=K_1}^{K_2-1}\mu_k^0 = n\sum_{k=K_1}^{K_2-1}\frac{\rho_k(t)}{k},
\]
while from \eqref{EtT12}
\[
 \E[\tX^2-\tX] = \sum_{k_1=K_1}^{K_2-1} \sum_{k_2=K_1}^{K_2-1}  \E \tT_{k_1,k_2} \sim \E[\tX]^2.
\]
It follows by the first and second moment methods that if $\E\tX\to 0$, then $\tX=0$ whp so $X=0$ whp,
and that if $\E\tX\to\infty$, then $\tX/\E\tX\pto 1$, and in particular that $\tX>0$ whp and so $X>0$ whp.%

The same argument applies, \emph{mutatis mutandis}, to higher moments: we obtain that for fixed $r$,
the $r$th factorial moment of $\tX$ is asymptotically $(\E\tX)^r$. In the case where $\E\tX=\Theta(1)$
this gives convergence in distribution to a Poisson distribution, and in particular that
$\Pr(X=0)=\Pr(\tX=0)+o(1)= e^{-\E\tX}+o(1)$. This gives the following result.%

\begin{lemma}\label{lmain}
Let $m=m(n)\in nI_0$, and set $t=t(n)=m/n$.
Let $\good$ be the event defined in \eqref{gooddef}, and suppose that $K_1(n)<K_2(n)=o(\sqrt{n}/\omega)$.
Set $\lambda(n)=\sum_{k=K_1}^{K_2-1} n \rho_k(t)/k$, 
and define $X$ as above to be the number of tree components of $G_m^{(n)}$
with between $K_1$ and $K_2-1$ vertices, and $Y$ the number of non-tree components
with between $K_1$ and $K_2-1$ vertices. Then
\[
 \Pr(\{X>0\} \cap \good) \le (1+o(1))\lambda(n),
\]
and
\begin{equation}\label{Ybd}
 \Pr(\{Y>0\} \cap \good) = o(\lambda(n)).
\end{equation}
Furthermore, if $\lambda(n)\to 0$ then
$X=Y=0$ whp, if $\lambda(n)\to\infty$ then $X/\lambda(n)\pto 1$ and $Y/\lambda(n)\pto 0$, and if $\lambda(n)=\Theta(1)$
then $\Pr(X=0)=e^{-\lambda(n)}+o(1)$ and $Y=0$ whp.
\end{lemma}
\begin{proof}
For the first statement, note that $\{X>0\}\cap \good$ implies $\tX>0$ and that,
by Lemma~\ref{lsofar}, $\E\tX\sim \lambda(n)$.
Then apply Markov's inequality to $\tX$.
For the second argue similarly for $Y$, recalling that $\E \tY=o(\E\tX)$ by Lemma~\ref{lsofar}.
The remaining statements for $X$ follow from the moment arguments above,
and those for $Y$ from Markov's inequality applied to $\tY$ and the fact that $Y=\tY$ whp.
\end{proof}
Recall that $\Pr(\good)\to 1$ by Lemma~\ref{iconc}. Hence Lemma~\ref{lmain} implies Theorem~\ref{th1}.%

\subsection{The largest small component}
Lemma~\ref{lmain} and a little calculation easily give the following result about the largest
component of $G_m$ of size $o(\sqrt{n})$.%

\begin{lemma}\label{lsmall}
Let $m=m(n)$. Set $t=t(n)=m/n$ and $\eps=\eps(n)=t-\tc$, and suppose that $|\eps|\le \eps_0$ and
that \eqref{epscond} holds.
Let $\omega(n)\to\infty$ slowly, and let $L_1^*$ denote the maximum $k\le \sqrt{n}/\omega^2$ such
that $G_m=G_m^{(n)}$ has a component with $k$ vertices. Then for any constant $x$ we have
\[
 \Pr\left( L_1^* \le \delta(t)^{-1}\left(\log(|\eps|^3n)-\frac{5}{2}\log\log(|\eps|^3n)+c(t)+x\right)  \right) \to e^{-e^{-x}}
\]
as $n\to\infty$, where $c(t)$ is a bounded function of $t$ defined in \eqref{ctdef}, and $\delta(t)$ is as
in Theorem~\ref{RW}.
\end{lemma}
\begin{proof}
Given $K=K(n)$ let
\begin{equation}\label{laKdef}
 \lambda_K=\lambda_{K,t} = \sum_{k\ge K} n\rho_k(t)/k.
\end{equation}
By Theorem~\ref{RW}, if $\delta(t)K\to\infty$ then
\begin{equation}\label{laK}
 \lambda_K \sim n\gamma(t) \sum_{k\ge K}k^{-5/2}e^{-\delta(t) k} \sim n \alpha(t) K^{-5/2} e^{-\delta(t) K},
\end{equation}
where $\alpha(t)=\frac{\gamma(t)}{1-e^{-\delta(t)}}$. Note that, from Theorem~\ref{RW},
$\delta(t)=\Theta(\eps^2)$ and so $\alpha(t)=\Theta(\delta(t)^{-1})=\Theta(\eps^{-2})$.%

Since $\delta(t)=\Theta(\eps^2)$ and $\eps^2\sqrt{n}/\log n\to\infty$ by \eqref{epscond}, if $\omega\to\infty$ sufficiently
slowly then \eqref{laK} gives $\lambda_{\sqrt{n}/\omega^2}=o(1)$. Taking $K_1=K$ and $K_2=\sqrt{n}/\omega^2$ in Lemma~\ref{lmain},
the quantity $\lambda(n)$ appearing there is $\lambda_K-\lambda_{\sqrt{n}/\omega^2}=\lambda_K-o(1)$, so it suffices
to check that setting
\[
 K = \delta(t)^{-1}\left(\log(|\eps|^3n)-\frac{5}{2}\log\log(|\eps|^3n)+c(t)+x\right),
\]
we have $\lambda_K=e^{-x}+o(1)$; the result then follows from Lemma~\ref{lmain}.
What remains is calculation; we outline an argument.%

Firstly note that since $x$ and $c(t)$ are bounded, \eqref{epscond} implies $K=o(\sqrt{n})$.
Taking logs, from \eqref{laK} our required condition $\lambda_K=e^{-x}+o(1)$ is equivalent to
\[
 K = \frac{1}{\delta}\left( \log\left(\frac{\alpha n}{K^{5/2}}\right) +x \right),
\]
where $\alpha=\alpha(t)$ and $\delta=\delta(t)$. This implies that $K=\Theta^*(\delta^{-1})=\Theta^*(\eps^{-2})$,
where, as usual, $\Theta^*$ notation hides factors of $\log n$. (Formally, we need to check that our final $K$ satisfies
this condition, but it does.)
It then follows that $\alpha n K^{-5/2} = \Theta^*(|\eps|^3 n)= n^{\Theta(1)}$, recalling \eqref{epscond}.
This implies that $\log(\alpha n K^{-5/2}) \sim \log(|\eps|^3 n)$.
Hence $K\sim \delta^{-1}\log(|\eps|^3 n)$ and
\[
 \frac{\alpha n}{K^{5/2}} \sim \frac{\alpha \delta^{5/2}{n}}{(\log(|\eps|^3n))^{5/2}}
  = \beta(t)\frac{|\eps|^3 n}{(\log(|\eps|^3 n))^{5/2}},
\]
where
\[
 \beta(t) =\frac{\gamma(t)\delta(t)^{5/2}}{1-e^{-\delta(t)}} |\eps|^{-3} = \Theta(1).
\]
Setting
\begin{equation}\label{ctdef}
 c(t)=\log \beta(t),
\end{equation}
the result follows easily.
\end{proof}%

At this point it is not hard to complete the proof of Theorem~\ref{th_subcrit} 
using (a simplified form of) an idea of Bollob\'as~\cite{BB84}, namely
establishing a gap in the sequence of component sizes. We postpone the proof to the
next section, where we treat the sub- and super-critical cases together.%

\section{The largest component}\label{sec_supcrit}
Our aim in this section is to prove Theorems~\ref{th_subcrit} and~\ref{th_supcrit};
we shall follow almost exactly the strategy of Bollob\'as~\cite{BB84},
using various estimates proved in the previous section in place of the corresponding (much simpler) formulae
for the Erd\H os--R\'enyi random graph $G_{n,m}$ (or $G_{n,p}$). Note that these estimates are not as precise
as those for $G_{n,m}$, so we end up with weaker results than Bollob\'as~\cite{BB84} and
{\L}uczak~\cite{Luczak90} proved for $G_{n,m}$.%

The first step is to show that in the supercritical case, we have roughly the right number of vertices
in `large' components; we start by studying the `scaling limit' $\rho(t)$.
Warnke and the second author~\cite{RWsci,RW12} showed that, defining
\[
 \rho(t)=1-\sum_{k=1}^\infty \rho_k(t),
\]
for any fixed $t$ we have $L_1(G_{\lfloor tn\rfloor})/n\pto \rho(t)$. 
In~\cite[Theorem 2.5]{RW17}, they showed
that $\rho(t)$, which is equal to $0$ on $[0,\tc]$, is analytic on $[\tc,\tc+\eps_0]$.
In particular, for $0<\eps\le \eps_0$ we have
\begin{equation}\label{rhosim2}
 \rho(\tc+\eps) = \xi\eps +O(\eps^2),
\end{equation}%
establishing \eqref{rhosim1}.

Following Bollob\'as~\cite{BB84}, we estimate the number of vertices in large components by considering
small components, primarily small tree components!%

In the following lemmas, we (implicitly) assume that \eqref{epscond} holds, and that $|\eps|\le \eps_0$.
Since $\eps^2\sqrt{n}/\log n\to\infty$, we can choose $\omega=\omega(n)$ so that $\omega\to\infty$ 
(as slowly as we like) but
\begin{equation}\label{eocond}
 \frac{\eps^2 \sqrt{n}}{\omega^2 \log n} \to\infty.
\end{equation}
We choose such an $\omega$ for the quantity appearing in \eqref{gooddef}. Set
\[ 
 K=K(n)=\sqrt{n}/\omega^2,
\]
and write
\[
 S_T = \sum_{k=1}^K k T_k \hbox{\quad and\quad} S_C = \sum_{k=1}^K k C_k
\]
for the numbers of vertices in tree (respectively non-tree) components of size at most $K$.%

As usual, we write $\tS_T$ and $\tS_C$ for the corresponding `cooked' random variables.
\begin{lemma}\label{esmall}
Let $m=m(n)$ and set $t=t(n)=m/n$ and $\eps=\eps(n)=t-\tc$.
Suppose that $0<\eps\le \eps_0$, and that \eqref{eocond} holds. Then
\begin{equation}\label{eST}
 \E[\tS_T] =  (1-\rho(t))n + O(\omega\eps^{-1}\sqrt{n})
\end{equation}
and
\begin{equation}\label{eSC}
 \E[\tS_C] = O(\eps^{-3}) = o(\eps^{-1}\sqrt{n}).
\end{equation}
\end{lemma}%
Note that \eqref{eocond} certainly implies that $\omega=o(\eps^2\sqrt{n})$, so 
the $O(\omega\eps^{-1}\sqrt{n})$ error term here is $o(\eps n)$.
\begin{proof}
By Theorem~\ref{RW}, we have
\begin{equation}\label{rhoK}
 \sum_{k>K} \rho_k(t) \sim \gamma(t) \sum_{k>K}k^{-3/2}e^{-\delta(t) k} =O\left(\delta(t)^{-1} K^{-3/2} e^{-\delta(t) K}\right) = o(n^{-100}),
\end{equation}
recalling that $\delta(t)=\Theta(\eps^2)$ and so, by \eqref{eocond}, $\delta(t)K/\log n\to\infty$.
Recalling the notational convention \eqref{appconv}, by Lemma~\ref{lsofar}, for $k\le K$ we have
\[
 \E[k\tT_k] = n\rho_k(t)(1+O(k\omega/\sqrt{n})).
\]
Summing, we see that
\begin{equation}\label{EST}
 \E[\tS_T] = n\sum_{k=1}^K \rho_k(t) + O(E) = n\left(1-\rho(t)-\sum_{k>K}\rho_k(t)\right) +O(E),
\end{equation}
where the error term is
\begin{equation}\label{Edef}
 E = \omega\sqrt{n} \sum_{k=1}^K k\rho_k(t).
\end{equation}
Now, by Theorem~\ref{RW} and the elementary estimate $\sum_{k=1}^\infty k^{-1/2}e^{-ak}=O(a^{-1/2})$,
\begin{equation}\label{Ebound}
 E = O\left(\omega\sqrt{n} \sum_{k=1}^K k^{-1/2} e^{-\delta(t) k}\right)
 = O\left( \omega\sqrt{n}\delta(t)^{-1/2} \right) = O(\omega \sqrt{n} \eps^{-1}).
\end{equation}
This, together with \eqref{EST} and \eqref{rhoK}, implies \eqref{eST}.
The argument for \eqref{eSC} is similar but simpler. 
From Lemma~\ref{lsofar} we have $\E[\tC_k] = O(k\rho_k(t))$, so
\[
 \E[\tS_C] = O\left(\sum_{k=1}^K k^2\rho_k(t) \right).
\]
Using Theorem~\ref{RW} and the elementary estimate $\sum_{k=1}^\infty k^{1/2}e^{-ak}=O(a^{-3/2})$,
this gives $\E[\tS_C]=O(\delta(t)^{-3/2})=O(\eps^{-3})$, and \eqref{eSC} follows
using \eqref{eocond} to compare the two error terms.
\end{proof}
Following Bollob\'as~\cite{BB84}, we establish concentration of $S_T+S_C$, and hence
of $N_{>K}(G_m^{(n)})=n-S_T-S_C$, by considering the second moment of $S_T$.
Unfortunately, to get a useful bound on the variance we need a stronger assumption on $\eps$ than
is needed elsewhere in the argument, namely that $\eps^6 n\to\infty$.
This certainly implies \eqref{epscond}. Note that, under this assumption,
we can choose $\omega\to\infty$ such that \eqref{eocond} and
\begin{equation}\label{second}
 \omega^2 = o(\eps^6 n)
\end{equation}
hold.
\begin{lemma}\label{conc}
Let $m=m(n)$, and set $t=t(n)=m/n$ and $\eps=\eps(n)=t-\tc$.
Suppose that $0<\eps\le \eps_0$, that $\eps^6n\to\infty$, and that
\eqref{eocond} and \eqref{second} hold. Then
\[
 N_{>K}(G_m^{(n)}) = \rho(t)n + \Op(\omega^{1/2}\eps^{-1/2}n^{3/4}).
\]
\end{lemma}
\begin{proof}
%
Let $T_{k_1,k_2}$ denote the number of ordered pairs $(T_1,T_2)$ of distinct tree components of $G_m$
where $T_i$ has $k_i$ vertices, and $\tT_{k_1,k_2}$ the corresponding cooked variable.
From \eqref{EtT12} and the fact that $k\mu_k^0=\rho_k(t)$, for $1\le k_1,k_2\le K$ we have
\[
 \E[k_1k_2\tT_{k_1,k_2}] = \bb{1+O((k_1+k_2)\omega/\sqrt{n})} n^2 \rho_{k_1}(t) \rho_{k_2}(t).
\]
Hence, using symmetry to absorb the $k_2$ term in the $O(\cdot)$ error term above into the $k_1$ term,
and defining $E$ as in \eqref{Edef}, we have 
\begin{eqnarray*}
 \E[\tS_T^2-\tS_T] &=& \sum_{k_1=1}^K\sum_{k_2=1}^K \E[ k_1k_2\tT_{k_1,k_2}] \\
 &=& n^2\left(\sum_{k_1=1}^K \rho_k(t)\right)^2 + O(\omega n^{3/2}) \left(\sum_{k_1=1}^K k_1\rho_{k_1}(t)\right)
 \sum_{k_2=1}^K \rho_{k_2}(t) \\
 &=& (\E[\tS_T]+O(E))^2 + O(nE) \\
 &=& \E[\tS_T]^2 + O(nE),
\end{eqnarray*}
recalling \eqref{EST}, noting that $\sum_{k=1}^K\rho_k(t)\le 1-\rho(t)\le 1$,
and using the facts that $\E[\tS_T]=O(n)$ and $E=o(\eps n)=O(n)$.
Using once again that $\E[\tS_T]=O(n)$, it follows that
\begin{eqnarray*}
 \Var[\tS_T] &=& \E[\tS_T^2-\tS_T] + \E[\tS_T] - \E[\tS_T]^2 \\
 &=& O(nE+n) = O(\omega n^{3/2}\eps^{-1}), 
\end{eqnarray*}
recalling \eqref{Ebound}.
By Chebyshev's inequality we thus have $\tS_T=\E[\tS_T]+\Op(\omega^{1/2}\eps^{-1/2}n^{3/4})$.
From Lemma~\ref{esmall} we have $\E[\tS_T]=(1-\rho(t))n+O(\omega\eps^{-1}\sqrt{n})$ and
$\E[\tS_C]=o(\eps^{-1}\sqrt{n})$. The latter implies (by Markov's inequality)
that $\tS_C=\op(\omega\eps^{-1}\sqrt{n})$. Since, by \eqref{eocond}, $\omega\eps^{-1}\sqrt{n}=o(\omega^{1/2}\eps^{-1/2}n^{3/4})$, it follows that
\[
 n- \tS_T-\tS_C = \rho(t) n+ \Op(\omega^{1/2}\eps^{-1/2}n^{3/4}).
\]
It remains only to note that $N_{>K}(G_m)=n-S_T-S_C$, and that
on the `good' event $\good$, which holds whp, we have $\tS_T=S_T$ and $\tS_C=S_C$.
\end{proof}
The next step follows another idea from Bollob\'as~\cite{BB84}, establishing a gap
in the sequence of component sizes. Here we have more room in the calculations
than in~\cite{BB84}, since we are further from the critical window. In this lemma
the conditions on $\eps_1$ and $\omega$ correspond exactly to \eqref{epscond}
and \eqref{eocond}.
\begin{lemma}\label{lgap}
Let $\eps_1=\eps_1(n)>0$ and $\omega=\omega(n)$ satisfy $0<\eps_1\le \eps_0$, $\eps_1n^{1/4}(\log n)^{-1/2}\to\infty$
and $\omega^2=o(\eps_1^2\sqrt{n}/\log n)$. Set $K=K(n)=\sqrt{n}/\omega^2$.
Then whp the process $(G^{(n)}_m)$ is such that, for every $0\le m\le (\tc-\eps_1)n$
and every $(\tc+\eps_1)n\le m\le (\tc+\eps_0)n$, the graph $G_m=G^{(n)}_m$ has no component
with between $K$ and $3K$ vertices.
\end{lemma}
\begin{proof}
By Theorem 2.8 of~\cite{RW17}, there is a constant $A$ such that whp
the graph $G_{(\tc-\eps_0)n}$ 
has no component with more than $A\log n$ vertices.
Hence whp none of the graphs $G_m$, $0\le m\le (\tc-\eps_0)n$,
(all subgraphs of $G_{(\tc-\eps_0)n}$)
has a component of size more than $A\log n\le K$. 
Suppose then that $m=m(n)$ satisfies the given conditions and, as usual,
define $t=t(n)=m/n$ and $\eps=\eps(n)=t-\tc$. Thus, by assumption, $|\eps|\ge \eps_1$.
Define $\lambda_{K,t}$ as in \eqref{laKdef}, ignoring the irrelevant rounding to integers.
Since $|\eps|\ge\eps_1$, our assumption on $\omega$ implies that \eqref{eocond} holds.
From \eqref{laK} and the fact that $\delta(t)=\Theta(\eps^2)$,
we see (considering only the $e^{-\delta(t)K}$ term), that
\begin{equation}\label{laK1}
 \lambda_{K,t}=o(n^{-100}),
\end{equation}
say. Recall the definition \eqref{gooddef} of the `good' event $\good=\good_{\le m}$. Note
that for $m\le m^+=(\tc+\eps_0)n$, if $\good_{\le m^+}$ holds, then so does $\good_{\le m}$.
Let $Z_m=X_m+Y_m$ denote the total number of components (tree plus non-tree)
of $G_m$ with between $K$ and $3K$ vertices. Since $3K=o(\sqrt{n}/\omega)$, by
the first two parts of Lemma~\ref{lmain} we have
\begin{eqnarray*}
 \Pr( \{Z_m>0\} \cap \good_{\le m^+}) &\le&  \Pr( \{Z_{m}>0\} \cap \good_{\le m})  \\
 &\le& (1+o(1))(\lambda_{K,t}-\lambda_{3K,t}) \\
 &\le& (1+o(1)) \lambda_{K,t} = o(n^{-100}).
\end{eqnarray*}
Recalling that $\Pr(\good_{\le m^+})\to 1$ and taking the union bound, the result follows.
\end{proof}
We are now ready to complete the proofs of our main results.
\begin{proof}[Proof of Theorem~\ref{th_subcrit}]
Let $m=m(n)$. Set $t=t(n)=m/n$ and $\eps=\eps(n)=t-\tc$, and suppose
that $-\eps_0\le \eps<0$ and that \eqref{epscond} holds.
Choose $\omega=\omega(n)\to\infty$ such that \eqref{eocond} holds,
and set $K=\sqrt{n}/\omega^2$. By Lemma~\ref{lgap}, whp no $G_{m'}$, $m'\le m$,
has a component of size between $K$ and $3K$. Since adding one edge to a graph
cannot more than double the size of the largest component, and $G_0$ consists of isolated vertices,
it follows that whp $G_m$ has no component with more than $K$ vertices.
The result now follows from Lemma~\ref{lsmall}.
\end{proof}
\begin{proof}[Proof of Theorem~\ref{th_supcrit}]
Once again we follow the strategy of Bollob\'as~\cite{BB84}.
Let $m=m(n)$ be such that, setting $t=t(n)=m/n$ and $\eps=\eps(n)=t-\tc$, we
have $0<\eps\le \eps_0$ and $\eps^6n\to\infty$.
As noted above, this implies that \eqref{epscond} holds,
and we may choose $\omega=\omega(n)\to\infty$ such that \eqref{eocond} and \eqref{second}
hold, and $\omega\le \log n$, say. 
Set $\hat\eps=\eps/2$, $\hat{t}=(\tc+t)/2=\tc+\hat\eps$, and $\hat{m}=\floor{\hat{t} n}$.
Of course,
\eqref{eocond} and \eqref{second} hold with $\hat\eps$ in place of $\eps$.
Let $K=K(n)=\sqrt{n}/\omega^2$, and
call a component of some $G_{m'}$ \emph{small} if it has fewer than $K$ vertices,
\emph{medium} if it has between $K$ and $3K$ vertices, and \emph{large} if it has
more than $3K$ vertices.
By Lemma~\ref{lgap}, whp the process $(G_i)_{i\ge 0}$ is such that for every $\hat{m}\le m'\le m$ there
are no medium components. Since two small components cannot unite to form a large
component, it follows that whp no new large components are created during
this part of the process.

\begin{claim}\label{Claim}
Whp $G_m$ has only one large component.
\end{claim}

Before proving this, let us see that Theorem~\ref{th_supcrit} follows. 
By Lemma~\ref{conc}, the number $N_{>3K}(G_m)$ of vertices of $G_m$ in large components
is $\rho(\tc+\eps) n+\Op(\omega^{1/2}\eps^{-1/2}n^{3/4})$.
In the light of Lemma~\ref{lsmall}, given the claim this establishes
Theorem~\ref{th_supcrit}, but with an extra factor $\omega^{1/2}$ in the error term in \eqref{newL1bd}.
We can remove this factor by observing that $\omega(n)$ does not appear in the statement
of Theorem~\ref{th_supcrit}, and can be taken to tend to infinity as slowly as we like.

It remains to prove Claim~\ref{Claim}.
For this, it is enough to condition on $G_{\hat{m}}$, and show that whp all large
components of $G_{\hat{m}}$ are contained in a single component of $G_m$.

By Lemma~\ref{conc} (now applied with $\hat\eps$ in place of $\eps$), 
and \eqref{rhosim1}, whp $G_{\hat{m}}$
has at least, say,
\[
 L = \rho(\hat{t}) n/2=\Omega(\hat\eps n)=\Omega(\eps n)
\]
vertices in large components. Let us condition
on $G_{\hat{m}}$, assuming that this holds.

For $i=0,\ldots,\log n$, set $m_i=\hat{m}+i(\eps n)/(2\log n)$, noting that $m_{\log n}=m$. (As usual,
we ignore the rounding to integers.)
Let $B_i$ be the event that the following hold: (i) $G_{m_{i-1}}$ has at least two large components,
and (ii) $G_{m_{i-1}}$ has a large component that is not connected in $G_{m_i}$ to some other
large component of $G_{m_{i-1}}$. We shall show in a moment that $\Pr(B_i)=o(1/\log n)$.
Assuming this, then whp none of $B_1,\ldots,B_{\log n}$ holds. Recalling that whp
no new large components are created during this part of the process,
it follows that whp each $G_{m_i}$ has either only a single large component,
or at most half as many large components as $G_{m_{i-1}}$. Since $G_{m_0}$ has
at most $n$ large components, it follows that whp $G_m$ has a single large component, as required.
It remains only to bound $\Pr(B_i)$. Recall that we are conditioning on $G_{m_0}$, and assuming
that it has at least $L=\Omega(\eps n)$ vertices in large components. Hence $G_{m_{i-1}}$ (a supergraph
of $G_{m_0}$) has at least $L$ vertices in large components. If $G_{m_{i-1}}$ has only one large
component, there is nothing to prove. Otherwise, consider some large component $C$ of $G_{m_{i-1}}$,
and let $R$ be the union of the other large components. Note that $|C|$, $|R|\ge 3K$,
and $|C|+|R|\ge L$, so $|C||R|\ge KL$.
There is a set $S$ of
$|C||R| \ge KL = \Omega(n^{1/2}\omega^{-2}\eps n) = \Omega(\eps n^{3/2}\omega^{-2})$ potential edges
$e$ such that the addition of any edge $e\in S$ would join $C$ to $R$. In any given step,
the conditional probability of adding some such edge is $\Omega(|S| n^{-2})$, say from
\eqref{pisol} and \eqref{pnisol}. Hence the probability that no such edge is added between
steps $m_{i-1}$ and $m_i$ is at most
\begin{equation}\label{nojoin}
 \bb{1-\Omega(|S| n^{-2})}^{m_i-m_{i-1}} \le \exp\bb{-\Omega(|S| n^{-2} (m_i-m_{i-1}))}.
\end{equation}
Now
\[
 |S|n^{-2}(m_i-m_{i-1}) =\Omega\bb{ \eps n^{3/2}\omega^{-2} n^{-2} \eps n/ \log n}
 = \Omega\bb{ \eps^2 n^{1/2} \omega^{-2}/\log n }.
\]
Recall that $\eps^6n\to\infty$, so $\eps^2n^{1/2} = (\eps^6n)^{1/3}n^{1/6}$ grows at
least as fast as $n^{1/6}$. Since $\omega\le\log n$, the final bound above grows
much faster than $\log n$.
Hence the probability \eqref{nojoin} is $o(n^{-100})$, say. Taking a union bound over the large components $C$
of $G_{m_{i-1}}$ we have $\Pr(B_i)=o(n^{-99})$, completing the proof of Claim~\ref{Claim}, and hence 
of Theorem~\ref{th_supcrit}.
\end{proof}%


\begin{thebibliography}{10}%



%
\bibitem{BBWcrt} 
S.~Bhamidi, A.~Budhiraja and X.~Wang,
The augmented multiplicative coalescent and critical dynamic random graph models,  
\emph{Probab. Theory Relat. Fields} {\bf 160} (2014), 733--796.%
%

\bibitem{BBWsub} 
S.~Bhamidi, A.~Budhiraja and X.~Wang,
Bounded-size rules: The barely subcritical regime,  
\emph{Combin. Probab. Comput.} {\bf 23} (2014), 505--538.%
%

\bibitem{BBWagg}
S.~Bhamidi, A.~Budhiraja and X.~Wang,
Aggregation models with limited choice and the multiplicative coalescent, 
\emph{Random Struct. Alg.} {\bf 46} (2015), 55--116.%
%

\bibitem{BF01}
T.~Bohman and A.~Frieze.
Avoiding a giant component.
\emph{Random Struct. Alg.}  {\bf 19} (2001), 75--85.%
%

\bibitem{BKgiant}
T.~Bohman and D.~Kravitz, 
Creating a giant component, 
\emph{Combin. Probab. Comput.} {\bf 15} (2006), 489--511.%
%

\bibitem{BB84} 
B.~Bollob\'as,
The evolution of random graphs,
\emph{Trans. Amer. Math. Soc.} {\bf 286} (1984), 257--274.%
%

\bibitem{rg_bp} 
B.~Bollob\'as and O.~Riordan, 
Random graphs and branching processes, in 
\emph{Handbook of large-scale random networks}, 
Bolyai Soc. Math. Stud {\bf 18}, 
B.~Bollob\'as, R. Kozma and D.~Mikl\'os eds (2009), pp. 15--115.%
%

\bibitem{BRsimple} 
B.~Bollob\'as and O.~Riordan, 
A simple branching process approach to the phase transition in $G_{n,p}$, 
\emph{Electron. J. Combinatorics} {\bf 19} (2012), P21 (8 pp.)%
%

\bibitem{DKP} 
M.~Drmota, M.~Kang and K.~Panagiotou,
Pursuing the giant in random graph processes  (2013). 
Available at\\ https://www.dmg.tuwien.ac.at/drmota/01-universal.pdf.%
%

\bibitem{ELMPSconn}
H.~Einarsson, J.~Lengler, F.~Mousset, K.~Panagiotou and A.~Steger,
Connectivity thresholds for bounded size rules, 
\emph{Ann. Appl. Probab.} {\bf 26} (2016), 3206--3250.%
%

\bibitem{ER60} 
P.~Erd\H os and A.~R\'enyi, 
On the evolution of random graphs, 
\emph{Magyar Tud. Akad. Mat. Kutat\'o Int. K\"ozl.} {\bf 5} (1960), 17--61.%
%

\bibitem{JRW18}
S.~Janson, O.~Riordan and L.~Warnke,
Sesqui-type branching processes,
\emph{Stochastic Processes and their Applications} {\bf 128} (2018), 3628--3655.%
%

\bibitem{JS12} 
S.~Janson and J.~Spencer, 
Phase transitions for modified Erd\H os--R\'enyi processes, 
\emph{Ark. Mat.} {\bf 50} (2012), no. 2, 305--329.%
%

\bibitem{KPconn}
M.~Kang and K.~Panagiotou, 
On the connectivity threshold of Achlioptas processes,
\emph{J. Comb.} {\bf 5 } (2014), 291--304.%
%

\bibitem{KPS} 
M.~Kang, W.~Perkins and J.~Spencer, 
The Bohman--Frieze process near criticality, 
\emph{Random Struct. Alg.} {\bf 43} (2013), 221--250.%
%

\bibitem{KPSconjecture}
M.~Kang, W.~Perkins and J.~Spencer, 
Erratum to \lq\lq The Bohman--Frieze process near criticality\rq\rq,
\emph{Random Struct. Alg.} {\bf 46} (2015), 801.%
%

\bibitem{KLSsub}
M.~Krivelevich, P.~Loh  and B.~Sudakov, 
Avoiding small subgraphs in Achlioptas processes, 
\emph{Random Struct. Alg.} {\bf 34} (2009), 165--195.%
%

\bibitem{KLSHam} 
M.~Krivelevich, E.~Lubetzky and B.~Sudakov, 
Hamiltonicity thresholds in Achlioptas processes, 
\emph{Random Struct. Alg.} {\bf 37} (2010), 1--24.%
%

\bibitem{Luczak90} 
T.~{\L}uczak, 
Component behavior near the critical point of the random graph process, 
\emph{Random Struct. Alg.} {\bf 1} (1990), 287--310.%
%

\bibitem{RWsci} 
O.~Riordan and L.~Warnke, 
Explosive percolation is continuous, 
\emph{Science} {\bf 333} (2011), 322--324.%
%

\bibitem{RW12} 
O.~Riordan and L.~Warnke, 
Achlioptas process phase transitions are continuous, 
\emph{Ann. Appl. Probab.} {\bf 22} (2012), 1450--1464.%
%

\bibitem{RWself} 
O.~Riordan and L.~Warnke, 
Achlioptas processes are not always self-averaging, 
\emph{Physical Review E} {\bf 86} (2012), 011129.
%


\bibitem{RWsub}
O.~Riordan and L.~Warnke, 
The evolution of subcritical {A}chlioptas processes, 
\emph{Random Struct. Alg.} {\bf 47} (2015), 174--203.%
%

\bibitem{RW13} 
O.~Riordan and L.~Warnke, 
Convergence of Achlioptas processes via differential equations with unique solutions, 
\emph{Combin. Probab. Comput.} {\bf 25} (2016), 154--171.%
%


\bibitem{RW17}
O.~Riordan and L.~Warnke, 
The phase transition in bounded-size Achlioptas processes (2017), to appear in \emph{Memoirs of the AMS}.
Available at https://arxiv.org/abs/1704.08714.%
%

\bibitem{SS} 
S.~Sen, 
On the largest component in the subcritical regime of the Bohman--Frieze process,  
\emph{Electron. Commun. Probab.} {\bf 21} (2016), paper no. 64, 15 pp.%
%

\bibitem{SW07} 
J.~Spencer and N.~Wormald, 
Birth control for giants, 
\emph{Combinatorica} {\bf 27} (2007), 587--628.%
%

\bibitem{W99} 
N.~Wormald, 
The differential equation method for random graph processes and greedy algorithms, 
in \emph{Lectures on Approximation and Randomized Algorithms} (M. Karonski and H.J. Pr\"omel, eds), pp. 73--155. PWN, Warsaw, 1999.%
%

\end{thebibliography}
\end{document}